\documentclass[11pt,reqno]{amsart}

\usepackage{amsmath}
\usepackage{amsthm}
\usepackage{amssymb}
\usepackage{amsfonts}
\usepackage{amsxtra}
\usepackage{fullpage}
\usepackage{amscd}

\let\nc\newcommand

\theoremstyle{plain}
\newtheorem{thm}{Theorem}
\newtheorem{prop}[thm]{Proposition}
\newtheorem{cor}[thm]{Corollary}
\newtheorem{lem}[thm]{Lemma}

\theoremstyle{definition}
\newtheorem{defn}[thm]{Definition}

\newtheorem{remark}[thm]{Remark}
\numberwithin{thm}{section}
\nc{\bdm}{\begin{displaymath}}
\nc{\edm}{\end{displaymath}}
\nc{\bthm}{\begin{thm}}
\nc{\ethm}{\end{thm}}
\nc{\blem}{\begin{lem}}
\nc{\elem}{\end{lem}}
\nc{\bcor}{\begin{cor}}
\nc{\ecor}{\end{cor}}
\nc{\beq}{\begin{equation}\label}
\nc{\eeq}{\end{equation}}
\nc{\bprop}{\begin{prop}}
\nc{\eprop}{\end{prop}}
\nc{\bdefn}{\begin{defn}}
\nc{\edefn}{\end{defn}}

\nc{\Z}{\mathbb{Z}}
\newcommand{\N}{\mathbb{N}}

\newcommand{\C}{\mathbb{C}}
\newcommand{\h}{\mathfrak{h}}

\renewcommand{\Im}{\mbox{\textrm{Im}}\,}

\newcommand{\Ker}{\mbox{\textrm{Ker}}\,}

\newcommand{\End}{\mathrm{End}\,}

\nc{\Hom}{\mathrm{Hom}}
\nc{\rank}{\textrm{rank} \,}
\nc{\ds}{\dots}

\let\mc\mathcal
\let\mf\mathfrak

\nc{\HW}{\bar{H}_{\mathbf{c}}(W)}
\nc{\HK}{\bar{H}_{\mathbf{c}}(K)}
\nc{\HtK}{\widetilde{H}_{\mathbf{c}}(K)}
\nc{\CMW}{\textsf{CM}_{\mbf{c}}(W)}
\nc{\CMK}{\textsf{CM}_{\mbf{c}}(K)}
\nc{\mbf}{\mathbf}
\nc{\LK}{\textsf{Irr}(K)}
\nc{\LW}{\textsf{Irr}(W)}
\nc{\Res}{\mathsf{Res} \, }
\nc{\Ind}{\mathsf{Ind} \, }

\nc{\cont}{\textrm{cont}}

\nc{\eWb}{\mathbf{e}_{W_b}}
\nc{\bz}{\mathbf{z}}
\nc{\msf}{\mathsf}
\nc{\bo}{\mbf{0}}
\nc{\Uone}{\mc{U}_{1,+}}
\nc{\rightsim}{\stackrel{\sim}{\longrightarrow}}
\nc{\wh}{\widehat}
\nc{\Irr}{\mathsf{Irr}}

\nc{\minusone}{-1}
\nc{\minustwo}{-2}
\nc{\p}{\partial}
\renewcommand{\H}{\mathsf{H}}
\nc{\s}{\mathfrak{S}}
\nc{\ZH}{\mathsf{Z}}
\DeclareMathOperator{\Tor}{\mathrm{Tor}}
\DeclareMathOperator{\Ext}{\mathrm{Ext}}
\DeclareMathOperator{\Spec}{\mathrm{Spec}}
\DeclareMathOperator{\Supp}{\mathrm{Supp}}

\newcommand{\Lmod}[1]{#1\text{-}{\mathsf{mod}}}
\newcommand{\grmod}[1]{#1\text{-}{\mathsf{grmod}}}
\nc{\NN}{\mathsf{N}}
\nc{\Coh}{\mathsf{Coh}}
\nc{\OmCh}{\Lambda}
\nc{\mOCh}{\Omega}
\nc{\id}{\mathrm{id}}
\nc{\op}{\mathrm{op}}
\nc{\red}{\mathrm{red}}
\nc{\reg}{\mathrm{reg}}
\nc{\sm}{\mathrm{sm}}
\nc{\sing}{\mathrm{sing}}
\renewcommand{\o}{\otimes}

\nc{\Cs}{\C^{\times}}
\nc{\ra}{\rightarrow}
\nc{\ba}{\mathbf{a}}
\nc{\bb}{\mathbf{b}}
\nc{\tit}{\textit}
\nc{\<}{\langle}
\nc{\bx}{\mathbf{x}}
\nc{\bp}{\mathbf{p}}
\nc{\blambda}{\boldsymbol{\lambda}}
\nc{\bc}{\mbf{c}}

\nc{\pa}{\partial}
\newcommand{\idot}{{\:\raisebox{2pt}{\text{\circle*{1.5}}}}}
\newcommand{\pgrmod}[1]{#1\text{-}{\mathsf{pgr}}}
\newcommand{\promod}[1]{#1\text{-}{\mathsf{pmod}}}
\nc{\gr}{\mathsf{gr}}
\renewcommand{\>}{\rangle}
\newcommand{\CM}{\mathsf{CM}}

\begin{document}

\title{Endomorphisms of Verma modules for rational Cherednik algebras}

\author{Gwyn Bellamy}
\address{School of Mathematics and Statistics, University of Glasgow, University Gardens, Glasgow G12 8QW}
\email{gwyn.bellamy@glasgow.ac.uk}

\begin{abstract}
\noindent We study the endomorphism algebra of Verma modules for rational Cherednik algebras at $t = 0$. It is shown that, in many cases, these endomorphism algebras are quotients of the centre of the rational Cherednik algebra. Geometrically, they define Lagrangian subvariaties of the generalized Calogero-Moser space. In the introduction, we motivate our results by describing them in the context of derived intersections of Lagrangians.  
\end{abstract}

\maketitle

\section{Introduction}

\subsection{} The quest for a direct bridge between the geometric world of Fukaya categories and the world of (algebraic or analytic) microlocal sheaves on a symplectic manifold is ongoing, and is currently the subject of intense research. Though considerable progress has recently been made, for instance with Tamarkin's seminal work on quantizations of Lagrangians \cite{TamarkinDisplace}, it seems that concrete, computable examples of some expected correspondence are still desirable to aid one's (or at least the author's) intuition. One abode where such computable, though non-trivial, examples live is that of rational Cherednik algebras, beginning for instance with results of \cite{KR}. The corresponding symplectic manifold is the generalized Calegero-Moser space, or often times, more appropriately, a symplectic resolution of this space. 

Though we have no idea what the appropriate definition of Fukaya category should be in this case, or its possible relation to microlocal sheaves on the generalized Calogero-Moser space $X$, we study in this article a shadow of such a hoped for relationship. For arguments sake, we take the objects of the Fukaya category, a $2$-category, to be Lagrangian (or more generally coisotropic) subvarieties and the hom spaces given by derived intersections. Then, in the shadow of the Fukaya category, the derived intersections are replaced by the $\Tor$ and $\Ext$ groups between the structure sheaves of the Lagrangians subvarieties i.e. by the (co)homology of the derived intersections. On the microlocal side, we consider not quantized sheaves on the space $X$, but those sheaves of $\mc{O}_X$-modules that are \textit{quantizable} i.e. that admit some quantization. One major advantage of working on the generalized Calogero-Moser space, even though it is a singular symplectic variety, is that it is equipped with a canonical quantization by virtue of the fact that it is the centre of the rational Cherednik algebra. 

Our original motivation for the current work was quite different, see section \ref{sec:motivate}.  

\subsection{Quantizable modules} Associated to each complex reflection group $(\h,W)$, are the rational Cherednik algebras, a family of algebras depending on a pair of parameter $(t,\bc)$, where $t \in \C$. When $t = 0$, the algebra $\H_{\mbf{c}}$ is a finite module over its centre $\ZH_{\mbf{c}}$. The algebra $\ZH_{\mbf{c}}$ is the coordinate ring of a symplectic variety $X_{\bc}$, and the geometry of this symplectic variety is intimately entwined with the representation theory of $\H_{\mbf{c}}$. When $t \neq 0$, the rational Cherednik algebra is strongly non-commutative and provides a canonical quantization of the symplectic variety $X_{\mbf{c}}$. For fixed $\bc$, one can also consider the $\C[\mbf{t}]$-algebra $\H_{\mbf{t},\bc}$, which is flat over $\mbf{t}$. 

A $\H_{\bc}$-module $M$ is said to be \tit{quantizable} if it can be extended to a $\H_{\mbf{t},\mbf{c}}$-module, flat over $\C[\mbf{t}]$. In fact, for what follows, the existence of an extension to the 3rd infinitesimal neighborhood of zero in $\Spec \C[\mbf{t}]$ suffices. Gabber's Theorem implies that the support $\Supp M \subset X_{\bc}$ of a quantizable module $M$ is coisotropic  i.e. defined by the vanishing of an involutive ideal. 

Let $M,N$ be quantizable left $\H_{\bc}$-modules and $M'$ a quantizable right $\H_{\bc}$-module. Then, by the construction described in section 3.1 of \cite{GBVGinzburg}, the graded vector spaces $\Ext_{\H_{\bc}}^{\idot}(M,N)$ and $\Tor^{\H_{\bc}}_{\idot}(M',N)$ carry a canonical differential, the \textit{virtual de Rham differential}, making each into a complex. The cohomology of $\Ext_{\H_{\bc}}^{\idot}(M,N)$, resp. $\Tor^{\H_{\bc}}_{\idot}(M',N)$, is commonly called the \textit{virtual de Rham cohomology} of the pair $(M,N)$, resp. the \textit{virtual de Rham homology} of $(M',N)$. In the context of symplectic geometry, virtual de Rham (co)homology plays a (conjectural) role in counting intersection numbers of Lagrangian intersections, see \cite{VirtualdeRham}. With regards to the possible relationship between virtual de Rham (co)homology and microlocal sheaves, see also \cite[Remark 7.7]{KashiwaraConstructability}.

A pair of (left or right) quantizable modules $M$ and $N$ are said to have \textit{smooth intersection} if $\Supp M \cap \Supp N$ is contained in the smooth locus of $X_{\bc}$. In this case, results of \cite{GBVGinzburg}, together with some basic Mortia theory imply:

\begin{prop}\label{prop:BG}
Assume that $M$ and $N$ are generically simple, quantizable left $\H_{\mbf{c}}$-modules, with smooth intersection and $M'$ a quantizable right $\H_{\mbf{c}}$-module such that $e M \simeq M' e$ as $\ZH_{\mbf{c}}$-modules. Then, 
\begin{enumerate}
\item The space $\Tor^{\idot}_{\H_{\mbf{c}}} (M',N)$ is a graded commutative algebra and $\Ext^{\idot}_{\H_{\mbf{c}}} (M,N)$ is a module over $\Tor^{\idot}_{\H_{\mbf{c}}} (M',N)$.
\item The virtual de Rham differential on $\Tor^{\idot}_{\H_{\mbf{c}}} (M',N)$ makes it into a BV(=Batalin-Vilkovisky)-algebra so that $\Ext^{\idot}_{\H_{\mbf{c}}} (M,N)$ becomes a BV-module over $\Tor^{\idot}_{\H_{\mbf{c}}} (M',N)$, again via the virtual de Rham differential.
\end{enumerate}
\end{prop}

The proof of Proposition \ref{prop:BG} is explained in section \ref{sec:BV}. The terms used in the statement of the proposition are defined there too. We refer the reader to the appendix for the definition of BV-algebra.

\subsection{Verma modules} Let $p \in \h^*$ and $\lambda \in \Irr (W_p)$. Then $\lambda$ is a $\C[\h^*] \rtimes W_p$-module, where $\C[\h^*]$ acts by evaluation at $p$. Since $\C[\h^*] \rtimes W_p$ is a subalgebra of $\H_{\mbf{c}}$, we can define the induced module $\Delta(p,\lambda) = \H_{\mbf{c}} \o_{\C[\h^*] \rtimes W_p} \lambda$. Since the definition of $\Delta(p,\lambda)$ makes sense over $\H_{\mbf{t},\mbf{c}}$, it is a quantizable module. It has recently been shown by Bonnaf\'e and Rouquier \cite{BonnafeRouquier} that each block $\Omega$ of the Calogero-Moser partition of $\Irr (W)$ has a canonical representative $\lambda_{\Omega}$. It is defined in terms of the $b$-invariant; see section \ref{sec:CM}. Now we can state the main result of this paper. 

\begin{thm}\label{thm:main}
Let $p \in \h^*$ and $\Omega$ a CM-partition for $(W_p, \mbf{c}')$. 
\begin{enumerate}
\item The Lagrangian module $\Delta(p,\lambda_{\Omega})$ is simple. 
\item The canonical map $\ZH_{\mbf{c}} \rightarrow \End_{\H_{\mbf{c}}}(\Delta(p,\lambda_{\Omega}))$ is surjective.
\item It is smooth if and only if $\Omega = \{ \lambda_{\Omega} \}$.
\end{enumerate} 
\end{thm}

When the support of $\Delta(p,\lambda_{\Omega})$ is contained in the smooth locus of $X_{\mbf{c}}$, one can show more. Namely, $E_{\ba,\Omega} := \End_{ \H_{\mbf{c}}}(\Delta(p,\lambda_{\Omega}))$ is a polynomial ring and $\Delta(p,\lambda)$ is a free $E_{\ba,\Omega}$-module of rank $|W|$ (here $\ba$ is the image of $p$ in $\h^* / W$). In general the Verma modules $\Delta(p,\lambda)$ are \textit{not} simple if $\lambda$ is not the canonical representative in its Calogero-Moser partition, and its endomorphism ring is non-commutative. 

As a consequence of Theorem \ref{thm:main} and Proposition \ref{prop:BG}, we deduce

\begin{cor}\label{cor:torext1}
Let $M$ be a quantizable module and $\Omega = \{ \lambda_{\Omega} \}$ a CM-partition for $(W_p, \mbf{c}')$. Then $\Tor_{\idot}^{\H_{\mbf{c}}} ((p,\lambda_{\Omega}^*)\Delta,M)$ is a BV-algebra and $\Ext^{\idot}_{\H_{\mbf{c}}} (\Delta(a,\lambda_{\Omega}),M)$ is a BV-module over the algebra $\Tor_{\idot}^{\H_{\mbf{c}}}  ((p,\lambda_{\Omega}^*)\Delta ,M)$.  
\end{cor}  

\subsection{Self-intersections} In particular, one can consider the derived self-intersections of the modules $\Delta(p,\lambda_{\Omega})$. We assume for the remainder of this section that $\Omega = \{ \lambda_{\Omega} \}$. Then the endomorphism ring $E_{\ba,\Omega}$ is a quotient $ \ZH_{\mbf{c}} / I$ of $\ZH_{\mbf{c}}$. Moreover, it is the coordinate ring of a smooth Lagrangian $\Lambda_{\ba,\Omega} \simeq \mathbb{A}^n$. This means that $\NN_{\ba,\Omega}^{\vee} := I / I^2 $ is a free $E_{\ba,\Omega}$-module. It is the module of sections of the conormal bundle of $\Lambda_{\ba,\Omega}$ in $X_{\mbf{c}}$. Its dual $\NN_{\ba,\Omega} := (I / I^2)^{\vee}$ is the module of sections of the normal bundle of $\Lambda_{\ba,\Omega}$ in $X_{\mbf{c}}$. The following is an application of the theory developed in \cite{GBVGinzburg}. 

\begin{cor}\label{thm:extiso1}
\begin{enumerate}
\item There is an isomorphism of Gerstenhaber algebras 
$$
\Tor_{\idot}^{\H_{\mbf{c}}}((p,\lambda_{\Omega}^*)\Delta,\Delta(p,\lambda_{\Omega})) \simeq \wedge^{\idot} \NN_{\ba,\Omega}^{\vee}.
$$ 
\item There is an isomorphism of Gerstenhaber modules $\Ext^{\idot}_{\H_{\mbf{c}}}(\Delta(p,\lambda_{\Omega}),\Delta(p,\lambda_{\Omega})) \simeq \wedge^{\idot} \NN_{\ba,\Omega}$ over the Gerstenhaber algebras in (1). 
\end{enumerate}  
\end{cor}

We outline the proof of Corollaries \ref{cor:torext1} and \ref{thm:extiso1} in section \ref{sec:BV}. The appendix contains a summary of the main results of \cite{GBVGinzburg} that are required in the article. The reader can also find the definition of Gerstenhaber algebra there. Corollary \ref{thm:extiso1} implies that the virtual de Rham cohomology $H_{\mathrm{vir}}^{\idot} (\Delta(p,\lambda_{\Omega}),\Delta(p,\lambda_{\Omega}))$ equals the usual de Rham cohomology of the conormal bundle of $\Lambda_{\ba,\Omega}$ in $X_{\mbf{c}}$; this is just $\C$ in degree zero since the space is contractible. Similarly, the virtual de Rham homology $H^{\mathrm{vir}}_{\idot} ((p,\lambda_{\Omega}^*)\Delta ,\Delta(p,\lambda_{\Omega}))$ is, up to a degree shift, also equal to the de Rham cohomology of the conormal bundle. See Corollary \ref{cor:virtual} for details. 

\subsection{}\label{sec:motivate} This paper was motivated by the close relationship, via the limit $t \rightarrow 0$ of Suzuki's functor \cite{VV}, between the rational Cherednik algebra associated to the symmetric group $\s_n$ and modules for the affine Lie algebra $\widehat{\mf{gl}}_{m,\kappa}$ at the critical level $\kappa = - m$. If $\lambda$ is a partition of $n$ with at most $m$ parts then it is easy to see that this functor sends the Weyl module $\mathbb{V}(\lambda)$ to the Verma module $\Delta(\lambda)$.  The endomorphism ring $\End_{\widehat{\mf{gl}}}(\mathbb{V}(\lambda))$ is commutative and can be the identified with the ring of functions on a certain moduli space of $GL_m$-opers on the formal disc. Moreover, the centre of $\widehat{\mf{gl}}_{m,\kappa}$ surjects on to $\End_{\widehat{\mf{gl}}}(\mathbb{V}(\lambda))$, see \cite[Theorem 9.6.1]{FrenkelBook}. This is a perfect analogue of our Theorem \ref{thm:main}. By the result of Frenkel and Teleman \cite{FrenkelTeleman}, this identification, at least when $\lambda = 0$, extends to an identification of $\Ext^{\idot}(\mathbb{V}(0),\mathbb{V}(0))$ with the space of differential forms on the moduli space; a result completely analogous to our Corollary \ref{thm:extiso1}. Their (much more sophisticated) arguments also rely crucially on the fact that the vacuum module $\mathbb{V}(0)$ can be quantized i.e. it exists for all levels $\kappa$. 

\section*{Acknowledgments}

The author is supported by the EPSRC grant EP-H028153.

\section{Rational Cherednik algebras at $t = 0$}\label{sec:RCA}

\subsection{Definitions and notation}\label{subsection:defns}

Let $(W,\h)$ be a complex reflection group, where $\h$ is the reflection representation for $W$, and let $\mathcal{S}(W)$ be the set of all complex reflections in $W$. For each $s \in \mathcal{S}(W)$, choose vectors $\alpha_s \in \mathfrak{h}$ and $\alpha_s^{\vee} \in \mathfrak{h}^*$ that span the one dimensional spaces $\Im (s - 1)|_{\mathfrak{h}}$ and $\Im (s - 1)|_{\mathfrak{h}^*}$ respectively. We normalize $\alpha_s$ and $\alpha_s^{\vee}$ so that $\alpha_s^{\vee}(\alpha_s) = 2$. Let $\mathbf{c} : \mathcal{S}(W) \rightarrow \C$ be a $W$-equivariant function. The \textit{rational Cherednik algebra at $t = 0$}, as introduced by Etingof and Ginzburg \cite{EG}, and denoted $\H_{\mbf{c}}$, is the quotient of the skew group algebra of the tensor algebra $T(\mf{h} \oplus \mf{h}^*) \rtimes W$ by the ideal generated by the relations $[x,x'] = [y,y'] = 0$ and
\beq{eq:rel}
[y,x] = \sum_{s \in \mathcal{S}} \mathbf{c}(s) \frac{x(\alpha_s)\alpha_s^\vee (y)}{\alpha_s^\vee(\alpha_s)} s, 
\eeq
for all $x,x' \in \h^*$ and $y,y' \in \h$. By the PBW property, there is an isomorphism of vector spaces $\H_{\mbf{c}} \simeq \C [\h] \otimes \C W \otimes \C [\h^*]$.

The trivial idempotent in $\C W$ is denoted $e$. Whenever $A$ is an algebra containing $\C W$, $E$ will denote the functor $\Lmod{A} \rightarrow \Lmod{e A e}$ given by multiplication by $e$. 

\subsection{The generalized Calogero-Moser Space}\label{sub:calogerospace}
The centre $\ZH_{\mathbf{c}}$ of $\H_{\mathbf{c}}$ is an affine domain over which $\H_{\mathbf{c}}$ is a finite module. We shall denote by $X_{\mathbf{c}} := \textrm{Spec}\, \ZH_{\mathbf{c}}$ the corresponding affine variety. The space $X_{\mathbf{c}}$ is called the \textit{generalized Calogero-Moser space} associated to the complex reflection group $W$ at parameter $\mathbf{c}$. By \cite[Proposition 4.15]{EG} we have inclusions $\C[\mathfrak{h}]^W \hookrightarrow \ZH_{\mathbf{c}}$ and $\C[\h^*]^W \hookrightarrow \ZH_{\mathbf{c}}$. These inclusions define surjective morphisms $\pi \, : \, X_{\mathbf{c}} \ra \h^*/W$ and $\varpi :  X_{\mathbf{c}} \ra \h/W$ respectively. Both $\pi$ and $\varpi$ are flat of relative dimension $\dim \h$. Write 
$$
\Upsilon \ : \   X_{\mathbf{c}} \longrightarrow \h^* / W \times \h/W
$$
for the product morphism $\Upsilon := \pi \times \varpi$. It is a finite, and hence closed, surjective morphism. By putting $x \in \h^*$ in degree one, $y \in \h$ in degree $-1$ and each $w \in W$ in degree zero, it is clear from the relations (\ref{eq:rel}) that $\H_{\mathbf{c}}$ is a $\Z$-graded algebra. This implies that $\ZH_{\mathbf{c}}$ is also $\Z$-graded. Thus, there exists a canonical $\Cs$-action on $X_{\mathbf{c}}$. The map $\Upsilon$ is $\Cs$-equivariant since $\C [\h]^W \o \C [\h^*]^W$ is a graded subalgebra of $\ZH_{\mathbf{c}}$. Proposition 4.15 of \cite{EG} implies that  

\begin{lem}\label{lem:basicprop}
Let $\pi^{-1}(\ba)$ denote the scheme-theoretic fiber of $\pi$ over $\ba \in \h^*/W$. Then,
\begin{enumerate}

\item The algebra $\ZH_{\mathbf{c}}$ is free of rank $|W|$ over $\C[\h]^W \o \C[\h^*]^W$, hence $X_{\mbf{c}}$ is Cohen-Macaulay. 

\item The algebra $\C[\pi^{-1}(\ba)]$ is a free $\C[\h]^W$-module of rank $|W|$.

\end{enumerate}
\end{lem}

The affine scheme $\pi^{-1}(\ba)$ is neither reduced nor irreducible. The generalized Calogero-Moser space $X_{\mbf{c}}$ has a natural Poisson structure, see \cite{EG}. 

\subsection{$\H_{\mbf{c}}$-modules and fixed points of the $\Cs$-action}

In this section we define Verma modules and recall some basic facts about the representation theory of rational Cherednik algebras. These can be found for instance in \cite{BellSRA}. The stabilizer of $p \in \h$, resp. $q \in \h^*$, under $W$ is denoted $W_p$, resp. $W_q$. 

\begin{defn}\label{defn:modules}
The \tit{Verma module} associated to $p$ and $\lambda \in \Irr (W_p)$ is the induced module 
$$
\Delta (p,\lambda) := \H_{\mbf{c}} \o_{\C[\h^*] \rtimes W_p} \lambda,
$$
where the action of $\C[\h^*]$ on $\lambda$ is via evaluation at $p$. 
\end{defn}

We write $\Delta (\lambda)$ for $\Delta (0,\lambda)$. Let $\mf{m}_{\bb} \lhd \C[\h]^W$ the maximal ideal corresponding to $\bb \in \h / W$. The \tit{baby Verma module} associated to $p$, $\lambda$ and $\bb$ is defined to be
$$
\Delta (p,\lambda,\bb) := \Delta (p,\lambda) / \mf{m}_{\bb} \cdot \Delta (p,\lambda).
$$ 
The module $\Delta(0,\lambda,\mbf{0})$ is the baby Verma module studied in \cite{Baby}. If $L$ is a simple $\H_{\mbf{c}}$-module then $\dim L \le |W|$, with equality if and only if the support of $L$, a closed point of $X_{\mbf{c}}$, is contained in the smooth locus. As noted above, the map $\Upsilon$ is $\Cs$-equivariant. Since the image of $0$ in $\h^* / W \times \h / W$ is the unique $\Cs$-fixed point of that space, the finitely many closed point of $\Upsilon^{-1}(0)$ are the precisely the $\Cs$-fixed points in $X_{\mbf{c}}$. The simple $\H_{\mbf{c}}$-modules supported at each of these fixed points is a graded $\H_{\mbf{c}}$-module. These simple graded modules are naturally parameterized by the set $\Irr (W)$. Namely, for each $\lambda \in \Irr (W)$, the module $\Delta(0,\lambda,\mbf{0})$ has a simple head $L(\lambda)$. The simple modules $L(\lambda)$ can be graded, $L(\lambda) \not\simeq L(\mu)$ for $\lambda \not\simeq \mu$, and they exhaust all simple, gradable $\H_{\mbf{c}}$-modules. If we stipulate that $L(\lambda)_0 = \lambda$ and $L(\lambda)_i = 0$ for all $i < 0$, then this fixes the grading on $L(\lambda)$.  

Our conventions about graded modules will be that $M[n]_i = M_{i - n}$. The coinvariant ring $\C[\h]^{co W} = \C[\h] / \langle \C[\h]^W_+ \rangle$ is a graded copy of the regular representation, considered as a $W$-module. The fake polynomial associated to $\lambda \in \Irr (W)$ is 
$$
f_{\lambda}(t) = \sum_{i \in \Z} [\C[\h]^{co W}_i : \lambda] t^i, 
$$
a polynomials in $\mathbb{N}[t]$. Define $b_{\lambda}$ to be the degree of the \textit{lowest} non-zero monomial. It is called the \textit{$b$-invariant} of $\lambda$. 

\section{Graded modules}\label{sec:CM}

Define $\H_{\bc}^+  = \H_{\bc} / \langle \C[\h^*]_+^W \rangle$ and, for each $\bb \in \h / W$, let $\H^{\bb}_{\bc}$ denote the quotient of $\H_{\bc}^+$ by the ideal generated by the maximal ideal defining $\bb$ in $\h / W$. Lemma \ref{lem:basicprop} implies that $\dim \H^{\bb}_{\bc} = |W|^3$. When $\bb = \mbf{0}$, the algebra $\H^{\mbf{0}}_{\bc}$ is the restricted rational Cherednik algebra and the modules $L(\lambda)$ defined previously are $\H^{\mbf{0}}_{\bc}$-modules. The blocks of this algebra define a partition, the \textit{Calogero-Moser partition} of $\Irr (W)$; $\Irr (W) = \bigcup_{i = 1}^k \Omega_i$, where $L(\lambda)$ and $L(\mu)$ belong to the same block of $\H^{\mbf{0}}_{\bc}$ if and only if there is some $\Omega_i$ containing both $\lambda$ and $\mu$. We denote the set of all $\Omega_i$ by $\CM_{\mbf{c}}(W)$. There is a natural bijection between $\CM_{\mbf{c}}(W)$ and the closed points of $\Upsilon^{-1}(0,0)$. 

\begin{lem}\label{lem:edecomp}
Let $\H^{\bb}_{\mbf{c}} = \bigoplus_{i = 1}^{\ell} B_i$ be the block decomposition of $\H^{\bb}_{\mbf{c}}$
\begin{enumerate}
\item For each $i$, there exists a unique simple module $L_i \in \Lmod{B_i}$ such that $e L_i \neq 0$. 
\item We have $\dim e L_i = 1$.
\item Let $P_i$ be the projective cover of $L_i$. Then, $\End_{\H^{\bb}_{\mbf{c}}} \left( \bigoplus_{i = 1}^{\ell} P_i \right) \simeq  e \H_{\mbf{c}}^{\bb} e$ and the functor $e : \Lmod{\H_{\mbf{c}}^{\bb}} \rightarrow \Lmod{e \H_{\mbf{c}}^{\bb}e}$ equals $\Hom_{\H^{\bb}_{\mbf{c}}} \left( \bigoplus_{i = 1}^{\ell} P_i , - \right)$.  
\end{enumerate}
\end{lem}

\begin{proof}
By M\"uller's Theorem \cite{Ramifications}, the blocks of $\H^{\bb}_{\mbf{c}}$ are in bijection with the closed points of $\Upsilon_{\mbf{c}}(\bb,0)$. Therefore, by \cite[Lemma 4.4.1]{smoothsra}, there exists at least one simple module $L \in \Lmod{B_i}$ such that $\dim e L = 1$. The functor $E : \Lmod{\H^{\bb}_{\mbf{c}}} \rightarrow \Lmod{e \H_{\mbf{c}}^{\bb} e}$, $E(M) = e M$, is a quotient functor with left adjoint ${}^{\perp} E$ given by ${}^{\perp} E(N) = \H_{\mbf{c}}^{\bb} e \o_{e \H_{\mbf{c}}^{\bb} e} N$. In particular, the adjunction $\mathrm{id} \rightarrow E \circ  {}^{\perp} E$ is an isomorphism. If $L'$ is another simple $B_i$-module with $e L \simeq e L'$ as $e \H_{\mbf{c}}^{\bb} e$-modules, then ${}^{\perp} E(L) \simeq {}^{\perp} E(L')$. This implies that both $L$ and $L'$ occur as composition factors of ${}^{\perp} E(L)$. But then
$$
1 = \dim e ({}^{\perp} E(L)) \ge [{}^{\perp} E(L) : L] + [{}^{\perp} E(L) : L'].
$$
This implies that $L \simeq L'$, confirming (1). 

Since $E$ is a quotient functor, it sends simple modules to simple modules. Therefore, since $e \H^{\bb}_{\mbf{c}} e$ is a commutative ring, $\dim e M$ is at most one-dimensional for any simple module $M$. This proves (2).  

Since $E$  is exact, Watt's Theorem says that there exists a projective $\H^{\bb}_{\mbf{c}}$-module $P$ such that $E \simeq \Hom_{\H^{\bb}_{\mbf{c}}}(P, - )$. The fact that $E(L_i) = e L_i$ and $E(L) = 0$ for all other simple $B_i$-modules implies that $P = \bigoplus_{i = 1}^{\ell} P_i$, proving (3).  
\end{proof}

\subsection{} It has been shown in \cite[Theorem 9.6.1]{BonnafeRouquier} that each $\Omega \in \CM_{\mbf{c}}(W)$ contains a unique representation with minimal $b$-invariant. This \textit{distinguished element} in $\Omega$ is denoted $\lambda_{\Omega}$. The following is Theorem 9.6.1 of \textit{loc. cit.} 

\begin{thm}\label{prop:Vfunctor}
Let $\Omega \in \CM_{\mbf{c}}(W)$. 
\begin{enumerate}
\item $\lambda_{\Omega}$ is the unique element in $\Omega$ such that $e L(\lambda_{\Omega}) \neq 0$. 
\item $\dim e L(\lambda_{\Omega}) = 1$. 
\item Let $B_{\Omega}$ be the block of $\H^{\mbf{0}}_{\mbf{c}}$ corresponding to $\Omega$. Then,  $\End_{\H^{\mbf{0}}_{\mbf{c}}}(P(\lambda_{\Omega})) = e B_{\Omega} e$. 
\item The centre of $\H^{\mbf{0}}_{\mbf{c}}$ surjects onto $\End_{\H^{\mbf{0}}_{\mbf{c}}}(\Delta(0,\lambda_{\Omega},\bo))$. 

\end{enumerate}
\end{thm}

\begin{remark}
The functor $E : \Lmod{\H^{\mbf{0}}_{\mbf{c}}} \rightarrow \Lmod{e\H^{\mbf{0}}_{\mbf{c}}e}$ is isomorphic to $\Hom_{\H^{\mbf{0}}_{\mbf{c}}}(\mathsf{P}_{\mbf{c}}, - )$, where $\mathsf{P}_{\mbf{c}} = \bigoplus_{\Omega \in \CM_{\mbf{c}}(W)} P(\lambda_{\Omega})$. Thus, it is strongly reminiscent of Soergel's $\mathbb{V}$-functor. Also, by choosing a linear character of the group $W$ other that the trivial representation, the analogue of Theorem \ref{prop:Vfunctor} still holds, but the distinguished element in each block will be different.  
\end{remark}

\subsection{} 

Let $\grmod{\H^+_{\mbf{c}}}$ denote the category of finitely generated, $\Z$-graded $\H^+_{\mbf{c}}$-modules, and similarly for $\grmod{\H^{\mbf{0}}_{\mbf{c}}}$, where in both cases homomorphisms are homogeneous of degree zero. Let $J_{\Z}(\H^{+}_{\mbf{c}})$ denote the graded Jacobson radical of $\H^{+}_{\mbf{c}}$. Since $\C[\h]^W_+$ is contained in $J_{\Z}(\H^{+}_{\mbf{c}})$, the inclusion $\grmod{\H^{\mbf{0}}_{\mbf{c}}} \hookrightarrow \grmod{\H^+_{\mbf{c}}}$ induces a bijection between isomorphism classes of simple graded modules. For each $\bb \in \h/W$, let $( - )^{\bb} : \Lmod{\H^+_{\bc}} \rightarrow \Lmod{\H^{\bb}_{\bc}}$ be the functor $M \mapsto M^{\bb} := \H_{\bc}^{\bb} \o_{\H^+_{\bc}} M$. To relate representations of $\H^+_{\bc}$ to representations of $\H^{\mbf{0}}_{\mbf{c}}$, let $\promod{\H^+_{\mbf{c}}}$, resp. $\pgrmod{\H^+_{\mbf{c}}}$, denote the category of projective $\H^+_{\mbf{c}}$-modules, resp. graded, projective $\H^+_{\mbf{c}}$-modules. These are exact categories. Then we have the following result, which is presumably  standard.  

\begin{lem}\label{claim:idempotentlifting} 
\begin{enumerate}
\item The abelian category $\grmod{\H^+_{\mbf{c}}}$ is Krull-Schmit. 
\item The simple module $L(\lambda)$ admits a projective cover $P^+(\lambda)$ in $\grmod{\H^+_{\mbf{c}}}$. 
\item The $P^+(\lambda)[m]$ for $m \in \Z$ are all indecomposable projectives in $\pgrmod{\H^+_{\mbf{c}}}$.
\item The functor $( - )^{\bb} : \promod{\H^+_{\mbf{c}}} \rightarrow \promod{\H^{\bb}_{\mbf{c}}(W)}$ is exact. 
\item The functor $( - )^{\bo} : \pgrmod{\H^+_{\mbf{c}}} \rightarrow \pgrmod{\H^{\mbf{0}}_{\mbf{c}}}$ is exact, essential surjective and induces a bijection between the blocks of $\H^+_{\mbf{c}}$ and the blocks of $\H^{\mbf{0}}_{\mbf{c}}$. 
\end{enumerate} 
\end{lem}

\begin{proof}
Since we were unable to find a suitable reference, we sketch the proof. As for the usual proof, e.g. \cite[Section 1.4]{Bensonfd}, of the Krull-Schmit property for Artinian rings, it suffices to show that the analogue of Fitting's Lemma holds in this setting. That is, given an  indecomposable module $M \in \grmod{\H^+_{\mbf{c}}}$ and $f \in \End_{\grmod{\H^+_{\mbf{c}}}}(M)$, we need to show that $M = \Ker (f^n) \oplus \mathrm{Im} (f^n)$ for some $n \gg 0$. Reading the usual proof of Fitting's Lemma e.g. Lemma 1.4.4 of \textit{loc. cit.}, we just need to show that there is some $n \gg 0$ such that $\Ker (f^{n+k}) = \Ker (f^n)$ and $\mathrm{Im}(f^{n+k}) = \mathrm{Im}(f^n)$ for all $k$. Let $A$ be the degree zero part of $\H_{\mbf{c}}^+$, a finite dimensional algebra. Since $M$ is finitely generated, there exists $N \gg 0$ such that $M = \H_{\mbf{c}}^+ \cdot M_{< N}$, and each weight space $M_i$ is a finite $A$-module. Therefore there is some $n \gg 0$ such that $\mathrm{Im}(f^{n+k}_{< N}) = \mathrm{Im}(f^n_{< N})$, where $f_{< N}$ is $f$ restricted to $M_{< N}$. This implies that $\mathrm{Im}(f^{n+k}) = \mathrm{Im}(f^n)$. To show that $\Ker (f^{n+k}) = \Ker (f^n)$ too, consider the short exact sequences $0 \rightarrow \Ker (f^{n}_i) \rightarrow M_i \rightarrow \mathrm{Im}(f^n_i) \rightarrow 0$ and $0 \rightarrow \Ker (f^{n+k}_i) \rightarrow M_i \rightarrow \mathrm{Im}(f^{n+k}_i) \rightarrow 0$ of $A$-modules. We certainly have $\Ker (f^{n+k}_i) \subset \Ker (f^{n}_i)$ and, by assumption, $\mathrm{Im}(f^n_i) = \mathrm{Im}(f^{n+k}_i)$. Since all modules have finite length, this implies that $\Ker (f^{n+k}_i) = \Ker (f^{n}_i)$ and hence $\Ker (f^{n+k}) = \Ker (f^{n})$. This proves part (1). 
 
Let $J_{\Z}(\H^+_{\bc})$ denoted the graded Jacobson radical of $\H^+_{\bc}$. The ideal generated by $\C[\h^*]_+^W$ is contained in $J_{\Z}(\H^+_{\bc})$. Therefore, the short exact sequence $0 \rightarrow \< \C[\h^*]_+^W \> \rightarrow \H^+_{\bc} \rightarrow \H^{\mbf{0}}_{\mbf{c}} \rightarrow 0$ restricts to $0 \rightarrow \< \C[\h^*]_+^W \> \cap (\H^+_{\bc})_0 \rightarrow (\H^+_{\bc})_0 \rightarrow (\H^{\mbf{0}}_{\mbf{c}})_0 \rightarrow 0$. Moreover, $ \< \C[\h^*]_+^W \> \cap (\H^+_{\bc})_0$ is contained in the Jacobson radical of $(\H^+_{\bc})_0$. Therefore every idempotent (automatically homogeneous of degree zero) in $\H^{\mbf{0}}_{\mbf{c}}$ lifts to an idempotent in $\H^+_{\bc}$. Let $a \in \H^{\mbf{0}}_{\mbf{c}}$ be a primitive idempotent, and $n(\lambda) \in \Z$, such that $(\H^{\mbf{0}}_{\mbf{c}} a) [n(\lambda)]$ is the projective cover of $L(\lambda)$ in $\grmod{\H^{\mbf{0}}_{\mbf{c}}}$. We denote by the same letter, the lift of $a$ to $\H^+_{\bc}$. Then the claim is that $(\H^{+}_{\mbf{c}} a) [n(\lambda)]$ is the projective cover of $L(\lambda)$ in $\grmod{\H^+_{\mbf{c}}}$. It is certainly a graded projective module mapping surjectively onto $L(\lambda)$. If it is not the projective cover, then it is decomposable. Each summand is free over $\C[\h]^W$, hence defines, after applying $( - )^{\bo}$, a non-zero direct summand of $(\H^{\mbf{0}}_{\mbf{c}} a) [n(\lambda)]$. But this contradicts the indecomposability of the latter module. Parts (2) and (3) follow.  

Since each $P \in \promod{\H^+_{\mbf{c}}}$ is a direct summand of a direct sum of copies of $\H_{\mbf{c}}^+$, it is a free $\C[\h]^W$-module. This implies that $( - )^{\bb}$ is exact. The fact that it maps projectives to projective is just the adjunction $\Hom_{\H^{\bb}_{\bc}}(P^{\bb},M) = \Hom_{\H^+_{\bc}}(P,M)$. This proves part (3). 

Parts (2) and (3) imply the first part of (4). For the final statement about blocks, we note by part (3) that it suffices to show that $\Hom_{\H^+_{\mbf{c}}}(P^+(\lambda),P^+(\mu)) \neq 0$ if and only if $\Hom_{\H^{\bo}_{\mbf{c}}}(P(\lambda),P(\mu)) \neq 0$. Since each $P^+(\lambda)$ is a direct summand of $\H_{\mbf{c}}^+$, it is a graded free $\C[\h]^W$-module. After fixing a grading on $P^+(\lambda)$ and $P^+(\mu)$, this makes $\Hom_{\H^+_{\mbf{c}}}(P^+(\lambda),P^+(\mu))$ into a graded $\C[\h]^W$-module. Let $\C_0 = \C[\h]^W / \C[\h]^W_+$. Then, the sequence 
$$
0 \rightarrow \C[\h]^W_+ \cdot P^+(\mu) \rightarrow P^+(\mu) \rightarrow \C_0 \o_{\C[\h]^W} P^+(\mu) = P(\mu) \rightarrow 0
$$
is exact, and apply $\Hom_{\H^+_{\mbf{c}}}(P^+(\lambda), - )$, we deduce that the space $\Hom_{\H^{\bo}_{\mbf{c}}}(P(\lambda),P(\mu))$ equals $\C_0 \o_{\C[\h]^W} \Hom_{\H^+_{\mbf{c}}}(P^+(\lambda),P^+(\mu))$. By the graded Nakyama lemma, this implies that $\Hom_{\H^+_{\mbf{c}}}(P^+(\lambda),P^+(\mu)) \neq 0$ if and only if $\Hom_{\H^{\bo}_{\mbf{c}}}(P(\lambda),P(\mu)) \neq 0$. 
\end{proof}

\subsection{}\label{sec:fiberspi} We use Lemma \ref{claim:idempotentlifting} to lift the statements of Lemma \ref{prop:Vfunctor} to $\Lmod{\H^+_{\mbf{c}}}$. 

\begin{prop}\label{prop:liftPE}
We have an isomorphism of exact functors $E \simeq \bigoplus_{\Omega \in \CM_{\mbf{c}}(W)} \Hom_{\H^+_{\mbf{c}}} (P^+(\lambda_{\Omega}), -  ) : \Lmod{\H^+_{\mbf{c}}} \rightarrow \Lmod{e \H^+_{\mbf{c}} e}$.
\end{prop}

\begin{proof}
It is a standard fact in graded ring theory that the forgetful functor $\pgrmod{\H^+_{\mbf{c}}} \rightarrow \promod{\H^+_{\mbf{c}}}$ is essentially surjective. Therefore part (3) of Lemma \ref{claim:idempotentlifting}  implies that there exist integers $n_{\lambda}$ such that $E \simeq \bigoplus_{\lambda \in \Irr(W)} \Hom_{\H_{\bc}^+}(P^+(\lambda)^{n_{\lambda}}, - )$. Then $n_{\lambda} = \dim E(L(\lambda))$. But $E(L(\lambda)) = 0$ if $\lambda$ is not the distinguished representative in its block and $\dim E(L(\lambda_{\Omega})) = 1$.
\end{proof}

\section{Endomorphism Algebras}\label{sec:endo}

\subsection{}  

Recall that the \textit{degrees} of a complex reflection group $(W,\h)$ is the multiset of degrees of some (any) choice of homogeneous, algebraically independent generators of $\C[\h]^W$. For brevity, we denote by $b_{\Omega}$, the $b$-invariant of $\lambda_{\Omega}^*$. 

\bthm\label{thm:deltaend}
Let $\Omega$ be a CM-partition and $\lambda_{\Omega}$ the distinguished representative. Then, 
\begin{enumerate}
\item The endomorphism ring $E_{\Omega}$ of the Verma module $\Delta(\lambda_{\Omega})$ is an $\N$-graded quotient of $\ZH_{\mbf{c}}$. In particular, it is commutative. 
\item $e\Delta(\lambda)$ is a cyclic $E_{\Omega}$-module. 
\item The graded character of $E_{\Omega}$ is
$$
\textrm{ch}_q (E_{\Omega}) = q^{-b_{\Omega}}f_{b_{\Omega}} (q) \prod_{i = 1}^n \frac{1}{1 - q^{d_i}},
$$
where $d_1, \dots, d_n$ are the degrees of $(W,\h)$. 
\end{enumerate}
\ethm

\begin{proof}
All the claims of the theorem follow more or less directly from Proposition \ref{prop:liftPE}. Let $B$ be the block of $\H_{\mbf{c}}^+$ corresponding to $\Omega$ and $P^+(\lambda_{\Omega})$ the projective cover of $L(\lambda_{\Omega})$ in $\grmod{\H_{\mbf{c}}^+}$. Applying $\Hom_{\H_{\mbf{c}}^+}(P^+(\lambda_{\Omega}), - )$ to the exact sequence $P^+(\lambda_{\Omega}) \rightarrow \Delta(\lambda_{\Omega}) \rightarrow 0$ gives a surjection of $e B e$-modules, $e B e \rightarrow e \Delta(\lambda_{\Omega}) \rightarrow 0$. Thus, $e \Delta(\lambda_{\Omega}) $ is a cyclic $eBe$-module. Since the action of $\ZH_{\mbf{c}}$ on  $e \Delta(\lambda_{\Omega}) $ factors through $eBe$, it is also a cyclic $\ZH_{\mbf{c}}$-module. 

Similarly, applying $\Hom_{\H_{\mbf{c}}^+}(-, \Delta(\lambda_{\Omega}) )$ to $P^+(\lambda_{\Omega}) \rightarrow \Delta(\lambda_{\Omega}) \rightarrow 0$ gives an exact sequence $0 \rightarrow E_{\lambda_{\Omega}} \rightarrow e \Delta(\lambda_{\Omega})$. Let $\ZH_{\Omega}$ denote the image of $\ZH_{\mbf{c}}$ in $E_{\Omega}$. Then we have inclusions $\ZH_{\Omega} \hookrightarrow E_{{\Omega}} \rightarrow e \Delta(\lambda_{\Omega})$. These can be considered morphisms of finitely generated, graded $\C[\h]^W$-modules. If $\C_0$ is the quotient $\C[\h]^W$ by its augmentation ideal then Lemma \ref{prop:Vfunctor} implies that morphism $\C_0 \o_{\C[\h]^W} \ZH_{\Omega} \rightarrow \C_0 \o_{\C[\h]^W} e \Delta(\lambda_{\Omega})$ is an isomorphism. Therefore, the graded Nakayama lemma implies that $\ZH_{\Omega} \rightarrow  e \Delta(\lambda_{\Omega})$ is surjective. Thus, $\ZH_{\Omega} \stackrel{\sim}{\longrightarrow} E_{{\Omega}} \stackrel{\sim}{\longrightarrow} e\Delta(\lambda_{\Omega})$. This implies parts (1) and (2). 

Recall that $b_{\Omega} := b_{\lambda_{\Omega}^*}$. As in the proof of \cite[Theorem 5.6]{Baby}, an easy calculation shows that    
$$
\textrm{ch}_q (E_{\Omega}) = \textrm{ch}_q \left( e \Delta(\lambda)[-b_{\Omega}] \right) = q^{-b_{\Omega}}f_{\Omega}(q) \prod_{i = 1}^n \frac{1}{1 - q^{d_i}}.
$$
\end{proof}

Lemma \ref{prop:Vfunctor} and Theorem \ref{thm:deltaend} are false if $\lambda$ is not the distinguished representative in its block.

\begin{remark}
When $(\h,W) = (\C^n, \s_n)$ and $\bc \neq 0$, we have 
$$
\mathrm{ch}_q( E_{\lambda}) = \frac{1}{H_{\lambda}(q)},
$$
where $H_{\lambda}(q)$ is the hook polynomial of the partition $\lambda$. 
\end{remark}

\subsection{} When $\Omega = \{ \lambda_{\Omega} \}$ consists of a single element, one can say a great deal more about the endomorphism ring $E_{\Omega}$.

\begin{lem}\label{lem:smoothsupport}
If $\Omega = \{ \lambda_{\Omega} \}$ then $\Supp ( \Delta (\lambda))$ is contained in the smooth locus of $X_{\mbf{c}}$.
\end{lem}

\begin{proof}
As mentioned in section \ref{sec:CM}, there is a natural bijection between the elements of $\CM_{\mbf{c}}(W)$ and the points in $\Upsilon^{-1}(0,0)$. Let $\bx_{\Omega}$ be the point corresponding to $\Omega$. The singular locus of $X_{\mbf{c}}$ is $\Cs$-stable. The module $\Delta (\lambda)$ is graded by putting $1 \o \lambda$ in degree zero. Then, all non-zero weight spaces are positive. Therefore, if $I$ is the annihilator of $\Delta (\lambda)$ in $\ZH_{\mbf{c}}$, then the quotient $\ZH_{\mbf{c}}/I$ is positively graded with degree zero part equal to $\C$. This implies that $\Supp \Delta(\lambda)$ is contained inside the attacking set $\{ x \in X_{\mbf{c}} \ | \ \lim_{t \rightarrow \infty} t \cdot x = \bx_{\Omega} \}$. The singular locus $X_{\mbf{c}}(W)_{\sing}$ of $X_{\mbf{c}}$ is $\Cs$-stable. Therefore, if $\Supp \Delta (\lambda) \cap X_{\mbf{c}}^{\sing} \neq \emptyset$ then $\bx_{\Omega}$ belongs to this intersection.  
\end{proof}

\begin{cor}\label{cor:polyring}
If $\Omega = \{ \lambda_{\Omega} \}$, then the commutative ring $E_{\Omega}$ is a graded polynomial ring.
\end{cor}

\begin{proof}
The algebra $E_{{\Omega}}$ is $\N$-graded and connected. Therefore, there is a unique graded, maximal ideal $\mf{m}$ in $E_{{\Omega}}$. The algebra $E_{{\Omega}}$ is a finite module over $\C[\h]^W$. Therefore, by \cite[Corollary 1.4.5]{Benson}, the Krull dimension of $E_{{\Omega}}$ is $\dim \h$. Let $T = (E_+ / E_+^2)^*$ be the tangent space at the unique closed $\Cs$-fixed point of $\Spec (E_{{\Omega}})$. If we can show that $\dim T = n$ then the corollary follows from standard results on graded rings e.g. \cite[Theorem 1, Appendix III]{SerreLocalAlgebra}. 

Let's show $\dim T = n$. It must have dimension at least $n$. Since $\Spec (E_{{\Omega}})$ is a closed subvariety of $X_{\mbf{c}}$, $T$ is a subspace of $T_{\bx_{{\Omega}}} X_\mbf{c}$. The fact that $\bx_{{\Omega}}$ is a $\Cs$-fixed point implies that the space $T_{\bx_{{\Omega}}} X_{\mbf{c}}$ is a $\Cs$-module and $T$ is a submodule. We decompose $T_{\bx_{{\Omega}}} X_{\mbf{c}} = T_- \oplus T_0 \oplus T_+$, where $T_-$ consists of all weight spaces of strictly negative weights etc. We have $T \subset T_+$. Since $\bx_{{\Omega}}$ is an isolated fixed point, it follows from \cite[Corollary 2.2]{BBFix} that $T_0 = 0$. By assumption, $\bx_{{\Omega}}$ is contained in the smooth locus of $X_{\mbf{c}}$. Therefore, $T_{\bx_{{\Omega}}} X_\mbf{c}$ is a symplectic vector space. The symplectic form on $T_{\bx_{{\Omega}}} X_{\mbf{c}}$ is $\Cs$-invariant. This implies that $\dim T_+ = \dim T_- = n$. Thus, $T \subset T_+$ implies that $\dim T \le n$ as required. 
\end{proof}

\begin{cor}
If $\Omega = \{ \lambda_{\Omega} \}$, then the $E_{\Omega}$-module $\Delta (\lambda_{\Omega})$ is free of rank $|W|$.
\end{cor}

\begin{proof}
Since we have shown in Corollary \ref{cor:polyring} that $E_{\Omega}$ is a polynomial ring, and $\Delta (\lambda_{\Omega})$ is a finitely generated $E_{\Omega}$-module, it suffices to show that $\Delta (\lambda_{\Omega})$ is projective. Let $\mf{m} \subset E_{\Omega}$ be a maximal ideal and consider the quotient $L = \Delta (\lambda_{\Omega}) / \mf{m} \cdot \Delta(\lambda_{\Omega})$. This space is non-zero since any endomorphism $\phi \in E_{\Omega}$ whose image in $E_{\Omega} / \mf{m} \cdot E_{\Omega}$ is non-zero induces a non-zero endomorphism of $L$. On the other hand, since $e \Delta(\lambda_{\Omega})$ is a cyclic $E_{\Omega}$-module, the map 
$$
E_{\Omega} / \mf{m} \cdot E_{\Omega} \longrightarrow e \Delta (\lambda_{\Omega}) / \mf{m} \cdot e \Delta(\lambda_{\Omega}) = e L
$$
is an isomorphism. In particular, $\dim e L = 1$. By Lemma \ref{lem:smoothsupport}, the support of $\Delta(\lambda_{\Omega})$, and hence of $L$ too, is contained in the smooth locus of $X_{\mbf{c}}$. Therefore, $L$ is simple and has dimension $|W|$. Thus, every fiber of $\Delta (\lambda_{\Omega})$ has dimension $|W|$ over $E_{\Omega}$. 
\end{proof}

\subsection{} 

We have shown in Theorem \ref{thm:deltaend} that $\ZH_{\mbf{c}}$ surjects on to $\End_{\H_{\mbf{c}}}(\Delta(\lambda))$. 

\begin{prop}\label{thm:someiso}
Assume that $X_{\mbf{c}}$ is smooth. Then, multiplication defines a graded isomorphism
\beq{eq:endceniso}
\ZH_{\mbf{c}} \bigm/ \sqrt{ \langle \C[\h]_+^W \rangle } \stackrel{\sim}{\longrightarrow} \End_{\H_{\mbf{c}}}\left( \bigoplus_{\lambda \in \Irr(W)} \Delta(\lambda) \right).
\eeq
\end{prop}

\begin{proof}
The supports of the modules $\Delta(\lambda)$ are also disjoint because they are precisely the attracting sets for the $\Cs$-action c.f. Lemma \ref{lem:smoothsupport}. Therefore $\End_{\H_{\mbf{c}}}\left(\bigoplus_{\lambda \in \Irr(W)} \Delta(\lambda)\right) \simeq \bigoplus_{\lambda \in \Irr(W)} \End_{\H_{\mbf{c}}}(\Delta(\lambda))$. Theorem \ref{thm:deltaend} (1) together with Corollary \ref{cor:polyring} implies that the support of $\Delta(\lambda)$ is reduced. Thus, $\sqrt{ \langle \C[\h]_+^W \rangle } \cdot \Delta(\lambda) = 0$ for all $\lambda \in \Irr (W)$. These facts imply the statement of the proposition. 
\end{proof}

\begin{remark}
Even when $X_{\mbf{c}}$ is smooth, one can show that $\pi^{-1}(0)$ is not reduced. In fact, it will be reduced if and only if $W \simeq \Z_m$ and $\mbf{c}$ generic. To see this, notice that Theorem \ref{thm:someiso} implies that $\ZH_{\mbf{c}} / \sqrt{ \langle \C[\h]_+^W \rangle }$ is positively graded. However, if $W$ has irreducible representations of dimension greater than one then $\ZH_{\mbf{c}} / \langle \C [\h^*]^W_+ \rangle$ has non-zero graded pieces of negative degree.  
\end{remark}

\section{Generalized Verma modules}\label{sec:factor1}

In this section we extend Theorem \ref{thm:deltaend} to Verma modules for which $p \neq 0$. This is done by showing that the endomorphism ring of $\Delta (p,\lambda)$ is isomorphic to the endomorphism ring of the $\H_{\mbf{c}'}(W_p)$-module $\Delta (\lambda)$, where $\mbf{c}'$ denotes the restriction of $\mbf{c}$ to the reflections in $W_p$ and $\H_{\mbf{c}'}(W_p)$ denotes the rational Cherednik algebra associated to $(\h,W_p,\mbf{c}')$. First, we recall a certain completion of $\H_{\mbf{c}}(W)$, as defined by Bezrukavnikov and Etingof \cite{BE}. Our presentation will be slightly different from \textit{loc. cit.}, since it will be useful for the applications in \cite{BellSchubert} that it agrees with Wilson's factorization of the Calogero-Moser space when $W$ is the symmetric group.

Let $p \in \h^*$ and let $\ba$ be the image of $p$ in $\h^* / W$. We denote by $\mf{m}_{\ba}$ for the maximal ideal of $\C[\h^*]^W$ corresponding to $\ba$. The completion of $\H_{\mbf{c}}(W)$ with respect to the ideal generated by $\mf{m}_{\ba}$ is denoted $\widehat{\H}_{\mbf{c}}(W)_{\ba}$. The algebra $\C[\h^*]/ \langle \mf{m}_{\ba} \rangle$ is finite dimensional, with closed points corresponding to the $W$-orbit of $p$. Assume this orbit consists of $\ell$ points $p_1 = p,p_2, \ds, p_{\ell}$. Then, there exist primitive idempotents $\tilde{e}_1, \ds, \tilde{e}_{\ell}  \in \C[\h^*]/ \langle  \mf{m}_{\ba}\rangle $ such that  $\C[\h^*]/ \langle  \mf{m}_{\ba} \rangle = \bigoplus_{i = 1}^{\ell} (\C[\h^*]/ \langle \mf{m}_{\ba} \rangle) \cdot \tilde{e}_i$. Hensel's Lemma, \cite[Corollary 7.5]{Eisenbud}, implies that the primitive idempotents in 
$$
\wh{\C[\h^*]}_{\ba} := \lim_{\leftarrow k} \C[\h^*]/ \langle \mf{m}_{\ba}^k \rangle 
$$ 
are precisely the lifts $e_i$ of the $\tilde{e}_i$. Therefore,
$$
\wh{\C[\h^*]}_{\ba} = \bigoplus_{i = 1}^{\ell} \wh{\C[\h^*]}_{\ba} \cdot e_i = \bigoplus_{i = 1}^{\ell} \wh{\C[\h^*]}_{p_i}
$$
with each $\wh{\C[\h^*]}_{\bp} \cdot e_i$ a local ring. For all $k \ge 0$, the PBW property for rational Cherednik algebras implies that 
$$
\C[\h^*] \cap \mf{m}_{\ba}^k \cdot \H_{\mbf{c}}(W) = \mf{m}_{\ba}^k \cdot \C[\h^*].
$$
Hence, we have embeddings
$$
\C[\h^*] / \langle \mf{m}_{\ba}^k \rangle \hookrightarrow \H_{\mbf{c}}(W) / \mf{m}_{\ba}^k \cdot \H_{\mbf{c}}(W)
$$
and taking the inductive limit, $\wh{\C[\h^*]}_{\ba} \hookrightarrow \wh{\H}_{\mbf{c}}(W)_{\ba}$, where we have used the fact that the functor of inverse limit is left exact. Therefore, we have $e_i \in \wh{\H}_{\mbf{c}}(W)_{\ba}$ for all $i$ and 
\beq{eq:matrixalg}
\wh{\H}_{\mbf{c}}(W)_{\ba} = \bigoplus_{i,j = 1}^{\ell} e_i \wh{\H}_{\mbf{c}}(W)_{\ba} e_j.
\eeq
Let $\wh{\H}_{\mbf{c}'}(W_{p_i})_{p_i}$ denote the completion of $\H_{\mbf{c}'}(W_{p_i})$ with respect to the maximal ideal $\mf{n}_{p_i}$ in $\C[\h^*]^{W_{p_i}}$. We write $\widetilde{\C[\h^*]}_{p_i}$ for the completion of $\C[\h^*]$ with respect to the maximal ideal $\mf{n}_{p_i}$ in order to distinguish it from $\wh{\C[\h^*]}_{p_i}$. Write also $W_{i,j}$ for the subset of $W$ consisting of all elements $w$ such that $w \cdot e_i = e_j$. Before proving the main result of the section we require some preparatory lemmata. 

\begin{lem}\label{lem:multiso}
Multiplication defines a vector space isomorphism 
$$
e_i (\C W_{i,j} \o \C[\h]) \widehat{\o} \widetilde{\C[\h^*]}_{p_j} e_j \rightsim  e_i \wh{\H}_{\mbf{c}}(W)_{b} e_j.
$$
and hence
$$
e_i \wh{\H}_{\mbf{c}}(W)_{\ba} e_j \o e_j \wh{\H}_{\mbf{c}}(W)_{\ba} e_k \twoheadrightarrow e_i \wh{\H}_{\mbf{c}}(W)_{\ba} e_k
$$
for all $1 \le i,j,k \le \ell$.
\end{lem}

\begin{proof}
By standard results, e.g. \cite[Section 7]{EGAI}, we have canonical isomorphisms $\wh{\C[\h^*]}_{\ba} \cdot e_j \simeq \wh{\C[\h^*]}_{p_j} \simeq \widetilde{\C[\h^*]}_{p_j}$. Therefore, multiplication is a well-defined map. The algebra $\wh{\H}_{\mbf{c}}(W)_{\ba}$ is filtered by putting $\C W \o \wh{\C[\h^*]}_{\ba}$ in degree zero and $\h^* \subset \C[\h]$ in degree one. The PBW property implies that the associated graded algebra is the skew group ring $(W \ltimes \C[\h]) \widehat{\o} \wh{\C[\h^*]}_{\ba}$ and the claim of the lemma on the level of associated graded spaces is clear. Since multiplication is a filtration preserving map it follows that it is also an isomorphism.

The second claim now follows from the fact that $\widetilde{\C[\h^*]}_{p_j} e_j w = w \widetilde{\C[\h^*]}_{p_k} e_k$ for all $w \in W_{j,k}$ and $W_{i,j} \cdot W_{j,k} = W_{i,k}$.
\end{proof}

\begin{lem}\label{lem:calculation1}
In $\wh{\H}_{\mbf{c}}(W)_{\ba}$ we have $[x, e_i]e_i = 0$,  for all $x \in \h^* \subset \C[\h], \ 1 \le i \le \ell$.
\end{lem}

\begin{proof}
Let $s \in W$ be a reflection. Recall the vectors $\alpha_s \in \h$ and $\alpha_s^\vee \in \h^*$ as defined in (\ref{subsection:defns}). It is shown in section 3.5 of \cite{Baby} that the functional $\alpha_s^\vee$ can be extend to a $\C$-linear operator on $\C[\h^*]$ by setting
$$
\alpha^\vee_s(ff') = \alpha^\vee_s(f)f' + f\alpha^\vee_s(f') - \lambda_s \frac{\alpha^\vee_s(f)\alpha^\vee_s(f')}{\alpha^\vee_s(\alpha_s)} \alpha_s, \quad \forall \ f,f' \in \C[\h^*]
$$
This operator satisfies  
\beq{eq:relationbig}
[x,f] = - \sum_{s \in S} \mbf{c}(s) \frac{x(\alpha_s) \alpha_s^\vee(f)}{\alpha^\vee_s(\alpha_s)} s \quad \forall \ x \in \h^*.
\eeq
It is also shown that $\alpha^\vee_s(f) = 0$ for all $s \in \mc{S}$ and $f \in \C[\h^*]^W$. Therefore $\alpha^\vee_s$ extends to an operator on $\wh{\C[\h^*]}_{\bp}$ such that relation (\ref{eq:relationbig}) holds for $f \in \wh{\C[\h^*]}_{\bp}$. Applying $\alpha_s^\vee$ to $e_i$ and using the fact that $e_i$ is an idempotent gives
$$
\alpha_s^\vee(e_i) \left( 1 - 2 e_i + \lambda_s \frac{\alpha_s^\vee(e_i)}{\alpha^\vee_s(\alpha_s)} \alpha_s \right) = 0.
$$
Multiplying by $e_i$ and using the fact that $\wh{\C[\h^*]}_{\bp} e_i \simeq \wh{\C[\h^*]}_{p_i}$, which is a domain, we must have
$$
e_i \alpha_s^\vee(e_i) = 0 \quad \textrm{ or } \quad e_i \alpha_s^\vee(e_i) = \frac{\lambda_s}{\alpha^\vee_s(\alpha_s)} \alpha_s^{-1} e_i.
$$
However, $\alpha_s$ is invertible in the local ring $\wh{\C[\h^*]}_{p_i}$ if and only if $\alpha_s(p_i) \neq 0$ if and only if $s \cdot e_i \neq e_i$. Therefore, multiplying the expression for $[x,e_i]$ given in (\ref{eq:relationbig}) on the right by $e_i$ gives zero. 
\end{proof}
 
\begin{prop}\label{prop:propcomplete}
For $i = 1, \ds ,\ell$:
\begin{enumerate}
\item we have an isomorphism of completed algebras
$$
\theta_i : \wh{\H}_{\mbf{c}'}(W_{p_i})_{p_i} \rightsim e_i \wh{\H}_{\mbf{c}}(W)_{\ba} e_i,
$$
\item the functor $e_i : \Lmod{\wh{\H}_{\mbf{c}}(W)_{\ba}} \ra \Lmod{\wh{\H}_{\mbf{c}'}(W_{p_i})_{p_i}}$, $M \mapsto e_i M$ is an equivalence of categories,
\item we have an isomorphism of commutative algebras
$$
\phi_i : Z(\wh{\H}_{\mbf{c}}(W)_{\ba}) \rightsim Z(\wh{\H}_{\mbf{c}'}(W_{p_i})_{p_i}).
$$
\end{enumerate}
\end{prop}

\begin{proof}
By Lemma \ref{lem:multiso} we can define a map 
$$
\theta_i : \wh{\H}_{\mbf{c}'}(W_{p_i})_{p_i} = (\C W_{i,i} \o \C[\h]) \widehat{\o} \widetilde{\C[\h^*]}_{p_i} \ra e_i \wh{\H}_{\mbf{c}}(W)_{\ba} e_i,
$$
by $f \mapsto e_i f e_i$, which is an isomorphism of topological vector spaces. Note that $e_i$ commutes with elements from $\widetilde{\C[\h^*]}_{p_i}$ and $W_{p_i} = W_{i,i}$. To show that this is an algebra morphism we must show that 
$$
e_i x_1 x_2 e_i = e_i x_1 e_i  x_2 e_i, \quad \forall \ x_1,x_2 \in \h^* \subset \C[\h] 
$$
and 
$$
e_i [y,x]_1 e_i = [e_i y e_i,e_i x e_i]_2, \quad \forall \ y \in \h, \ x \in \h^*,
$$
where $[- , - ]_1$ denotes the commutator in $\wh{\H}_{\mbf{c}'}(W_{p_i})_{p_i}$ and $[ - , - ]_2$ the commutator in $\wh{\H}_{\mbf{c}}(W)_{\ba}$. The first equality follows from Lemma \ref{lem:calculation1} and the second follows directly from the relations (\ref{eq:rel}), noting that $e_i$ commutes with $y$. 

The decomposition (\ref{eq:matrixalg}) of $\wh{\H}_{\mbf{c}}(W)_{\ba}$ allows us to think of $\wh{\H}_{\mbf{c}}(W)_{\ba}$ as a ``matrix algebra''. In particular, the centre of $\wh{\H}_{\mbf{c}}(W)_{\ba}$ is contained in $\bigoplus_{i = 1}^{\ell} e_i \wh{\H}_{\mbf{c}}(W)_{\ba} e_i$ and the projection map $u_i : \wh{\H}_{\mbf{c}}(W)_{\ba} \rightarrow e_i \wh{\H}_{\mbf{c}}(W)_{\ba} e_i$ induces an isomorphism $u_i : Z(\wh{\H}_{\mbf{c}}(W)_{\ba}) \rightarrow Z(e_i \wh{\H}_{\mbf{c}}(W)_{\ba} e_i)$. Therefore, we have
$$
\phi_i = \theta_i^{-1} \circ u_i : Z(\wh{\H}_{\mbf{c}}(W)_{\ba}) \rightsim Z(\wh{\H}_{\mbf{c}'}(W_{p_i})_{p_i}).
$$
\end{proof}

\begin{cor}\label{cor:factorChere}
Let $p \in \h^*$ and $\ba \in \h^*/W$ as above. Let $\ba_i$ be the image of $p_i$ in $\h^* / W_{p_i}$. Then, we have an isomorphism of schemes $\pi^{-1}_W(\ba) \simeq \pi^{-1}_{W_{p_i}}(\ba_i)$.
\end{cor}

\begin{proof}
Since the isomorphism $\theta_i$ of Proposition \ref{prop:propcomplete} maps the space $\mf{n}_{p_i}$ onto $e_i \mf{m}_{\ba} e_i$, the map $\phi_i$ of Proposition \ref{prop:propcomplete} satisfies
$$
\phi_i(\mf{m}_{\ba} \cdot Z(\wh{\H}_{\mbf{c}}(W)_{\ba})) = \mf{n}_{p_i} \cdot Z \left( \wh{\H}_{\mbf{c}}(W_{p_i})_{p_i} \right).
$$
It is proved in \cite[Lemma 3.9]{Cuspidal} that $Z(\wh{\H}_{\mbf{c}}(W)_{\ba})$ is the completion of $Z(\H_{\mbf{c}}(W))$ with respect to the ideal generated by $\mf{m}_{\ba}$ and, similarly, that $Z(\wh{\H}_{\mbf{c}'}(W_{p_i})_{p_i})$ is the completion of $Z(\H_{\mbf{c}'}(W_{p_i}))$ with respect to the ideal generated by $\mf{n}_{p_i}$. Therefore, $\phi_i$ induces an isomorphism of commutative algebras
$$
\frac{Z(\H_{\mbf{c}}(W))}{\mf{m}_{b} \cdot Z(\H_{\mbf{c}}(W))} \rightsim \frac{Z(\H_{\mbf{c}'}(W_{p_i}))}{\mbf{n}_{p_i} \cdot Z(\H_{\mbf{c}'}(W_{p_i}))}.
$$
\end{proof}

\subsection{} The map $x \mapsto x$, $w \mapsto w$ and $y \mapsto y + y(p)$ for all $x \in \h^* \subset \C[\h]$, $w \in W_p$ and $y \in \h \subset \wh{\C[\h^*]}_p$ defines an isomorphism $\wh{\H}_{\mbf{c}'}(W_p)_p \rightsim \wh{\H}_{\mbf{c}'}(W_p)_0$. Therefore, we will think of the functor $e_1 : \Lmod{\wh{\H}_{\mbf{c}}(W)_{\ba}} \ra \Lmod{\wh{\H}_{\mbf{c}'}(W_{p})_{p}}$ as an equivalence 
$$
\Phi : \Lmod{\wh{\H}_{\mbf{c}}(W)_{\ba}} \rightsim \Lmod{\wh{\H}_{\mbf{c}'}(W_p)_{0}}.
$$
Now consider the generalized Verma module $\Delta (p,\lambda)$. Since $\mf{m}_{\ba} \cdot (1 \o \lambda) = 0$ and $\mf{m}_{\ba}$ is central, $\Delta (p,\lambda)$ is a $\wh{\H}_{\mbf{c}}(W)_{\ba}$-module. 

\begin{lem}\label{lem:PhiVerma}
Let $p \in \h^*$ and $\lambda \in \Irr (W_p)$, then $\Phi(\Delta (p,\lambda)) \simeq \Delta (\lambda)$.
\end{lem}

\begin{proof}
As a $\wh{\H}_{\mbf{c}}(W)_{\ba}$-module, 
$$
\Delta (p,\lambda) = \wh{\H}_{\mbf{c}}(W)_{\ba} \o_{\wh{\C[\h^*]}_{\ba} \rtimes W_p} \lambda = \bigoplus_{i,j = 1}^{\ell} e_i \wh{\H}_{\mbf{c}}(W)_{\ba} e_j \o_{\wh{\C[\h^*]}_{\ba} \rtimes W_p} \lambda.
$$
Recall that $\wh{\C[\h^*]}_{\ba}$ acts on $\lambda$ by evaluation at $p$. Therefore, $e_i \cdot \lambda = 0$ for all $i \neq 1$ and 
$$
\wh{\C[\h^*]}_{\ba} \rtimes W_p = \wh{\C[\h^*]}_{\ba} e_1 \rtimes W_p \oplus \left( \bigoplus_{i = 2}^{\ell} \wh{\C[\h^*]}_{\ba} e_i \right) \rtimes W_p
$$ 
implies that 
$$
\Delta (p,\lambda) = \wh{\H}_{\mbf{c}}(W)_{\ba} \o_{\wh{\C[\h^*]}_{\ba} \rtimes W_p} \lambda = \bigoplus_{i=1}^{\ell} e_i \wh{\H}_{\mbf{c}}(W)_{\ba} e_1 \o_{\wh{\C[\h^*]}_{\ba} e_1 \rtimes W_p} \lambda.
$$
Thus,
$$
e_1 \cdot \Delta (p,\lambda) = e_1 \wh{\H}_{\mbf{c}}(W)_{\ba} e_1 \o_{\wh{\C[\h^*]}_{\ba} e_1 \rtimes W_p} \lambda.
$$
This implies that $\Phi(\Delta (p,\lambda)) \simeq \Delta (\lambda)$.
\end{proof}

\subsection{The proof of Theorem \ref{thm:main}}

Since we have yet to define a simple quantizable module (this is done properly in section \ref{sec:BV}), we simple note that in this case it means that $e \Delta(p,\lambda_{\Omega})$ is a cyclic $\ZH_{\mbf{c}}$-module. We define $E_{\ba,\Omega} := \End_{\H_{\mbf{c}}}(\Delta (p,\lambda_{\Omega}))$. Since we can identify 
$$
\End_{\H_{\mbf{c}}}(\Delta (p,\lambda_{\Omega})) = \End_{\wh{\H}_{\mbf{c}}}(\Delta (p,\lambda_{\Omega})),
$$
Theorem \ref{thm:main} follows from Theorem \ref{thm:deltaend}, Proposition \ref{prop:propcomplete} (2) and Lemma \ref{lem:PhiVerma}. 

Again, if $\Omega = \{ \lambda_{\Omega} \}$ in $\CM_{\mbf{c}'}(W_p)$, then the algebra $E_{\ba,\Omega}$ is a polynomial ring of dimension $\dim \h$. Let $w \in W$, then $\Delta(p,\lambda) \simeq \Delta(w(p), w(\lambda))$, where $w(\lambda)$ is the representation of $W_{w(p)}$ corresponding to $\lambda$ under the isomorphism $w : W_{p} \rightsim W_{w(p)}$ of conjugation. Therefore, if $\ba$ is the image of $p$ in $\h^*/W$, then we denote by $\OmCh_{\ba,\Omega}$ the support of the $\ZH_{\mbf{c}}$-module $\Delta (p,\lambda)$, thought of as a subscheme of $X_{\mbf{c}}$. If $\Omega = \{ \lambda_{\Omega} \}$, then $\OmCh_{\ba,\Omega}$ is a smooth Lagrangian subvariety of $X_{\mbf{c}}$, isomorphic to $\mathbb{A}^n$, as a closed subscheme of $X_{\mbf{c}}$. The varieties $\Lambda_{\ba,\Omega}$ play a key role in \cite{BellSchubert}.

\subsection{An equivalence of categories}

Let $\Lmod{\H_{\mbf{c}}}_{\ba,{\Omega}}$ denote the category of finitely generated $\H_{\mbf{c}}$-modules scheme-theoretically supported on $\OmCh_{\ba,{\Omega}}$ i.e. those modules $M$ such that $I \cdot M = 0$, where $I$ is the ideal defining $\OmCh_{\ba,{\Omega}}$. The category of coherent $\mc{O}_{\OmCh_{\ba,{\Omega}}}$-modules is denoted $\Coh (\OmCh_{\ba,{\Omega}})$. In this section we prove the following theorem:

\begin{thm}\label{thm:equivalence}
If $\Omega = \{ \lambda_{\Omega} \}$, then the functor  
$$
F : \Lmod{\H_{\mbf{c}}}_{\ba,{\Omega}} \rightarrow \Coh (\OmCh_{\ba,{\Omega}}), \quad F(M) = \widetilde{\Hom_{\H_{\mbf{c}}}(\Delta (p,\lambda), M)},
$$
is an equivalence of categories with quasi-inverse $N \mapsto G(N) := \Delta (p,\lambda) \o_{\ZH_{\mbf{c}}} \Gamma(\OmCh_{\ba,{\Omega}},N)$. 
\end{thm}

\begin{lem}\label{lem:obequiv}
Let $e : \Lmod{\H_{\mbf{c}}} \ra \Coh (X_\mbf{c}(W))$ be the functor $M \mapsto \widetilde{e M}$.
\begin{enumerate}
\item The functor $e$ is an equivalence if and only if $X_{\mbf{c}}$ is smooth.
\item If $\Omega = \{ \lambda_{\Omega} \}$, then $e$ defines an equivalence $\Lmod{\H_{\mbf{c}}}_{\ba,{\Omega}} \rightsim \Coh (\OmCh_{\ba,{\Omega}})$.
\end{enumerate}
\end{lem}

\begin{proof}
Part (1) is well-known e.g. \cite{EG}.

Part (2): by Lemma \ref{lem:smoothsupport}, the assumption of part (2) implies that $\OmCh_{\ba,{\Omega}} \subseteq X_{\mbf{c}}^{\sm}$. Since $\OmCh_{\ba,{\Omega}} \cap X_{\mbf{c}}^{\mathrm{sing}} = \emptyset$, Hilbert's Nullstellensatz implies that $I(\OmCh_{\ba,{\Omega}}) + I(X_{\mbf{c}}^{\mathrm{sing}}) = \ZH_{\mbf{c}}$ and we can find a characteristic function $f \in \ZH_{\mbf{c}}$ taking the value $1$ at all points of $\OmCh_{\ba,{\Omega}}$ and vanishing on $X_{\mbf{c}}^{\mathrm{sing}}$. Replacing $\H_{\mbf{c}}$ by its localization at $f$ and $\ZH_{\mbf{c}}$ by its localization, we may assume that $\ZH_{\mbf{c}}$ is a regular affine algebra and $\H_{\mbf{c}}$ an Azumaya algebra over $\ZH_{\mbf{c}}$. We remark that part (1) still holds after localization. Recall that the centre $Z(\mc{A})$ of an abelian category $\mc{A}$ is defined to be the ring of endomorphisms $\End_{\mc{A}}(\id_{\mc{A}})$ of the identity functor. For a Noetherian $k$-algebra $A$, the centre of $\Lmod{A}$ is canonically isomorphic to the centre of $A$. The equivalence $e$ induces an isomorphism $Z(\Lmod{\H_{\mbf{c}}}) \rightsim Z(\Coh (X_\mbf{c}))$ such that the composite
$$
Z(\H_{\mbf{c}}) \rightsim Z(\Lmod{\H_{\mbf{c}}}) \rightsim Z(\Coh (X_{\mbf{c}})) \rightsim Z(\H_{\mbf{c}})
$$
is just the identity map. Let $I = I(\OmCh_{\ba,{\Omega}})$. We can identify $\Lmod{\H_{\mbf{c}}}_{\ba,{\Omega}} = \{ M \in \Lmod{\H_{\mbf{c}}} \ | \ i_M = 0 \ \forall \ i \in I \}$, where $i_M \in \End_{\H_{\mbf{c}}}(M)$ is the endomorphism defined by $i \in Z(\Lmod{\H_{\mbf{c}}})$. Similarly, $ \Coh (\OmCh_{\ba,{\Omega}}) = \{ \mc{F} \in \Coh (X_{\mbf{c}}) \ | \ i_{\mc{F}} = 0 \ \forall \ i \in I \}$. From this it follows that $e : \Lmod{\H_{\mbf{c}}}_{\ba,{\Omega}} \rightsim \Coh (\OmCh_{\ba,{\Omega}})$.
\end{proof}

Then, Theorem \ref{thm:equivalence} follows from

\begin{lem}\label{prop:progen}
The Verma module $\Delta (p,\lambda)$ is a projective generator in $\Lmod{\H_{\mbf{c}}}_{\ba,{\Omega}}$. 
\end{lem}

\begin{proof}
By Lemma \ref{lem:obequiv} it suffices to show that $\widetilde{e \Delta (p,\lambda)}$ is a projective generator of $\Coh (\OmCh_{\ba,{\Omega}})$. But, by Theorem \ref{thm:main}, $e \Delta (p,\lambda)$ is the regular representation as a $\C[\OmCh_{\ba,{\Omega}}]$-module. Therefore $\widetilde{e \Delta (p,\lambda)} \simeq \mc{O}_{\OmCh_{\ba,{\Omega}}}$ as sheaves on $\OmCh_{\ba,{\Omega}}$. 
\end{proof}

\begin{remark}
If $I$ is the ideal of $\ZH_{\mbf{c}}$ defining $\OmCh_{\ba,{\Omega}}$, then set $\H_{\mbf{c}}(\ba,{\Omega}) := \H_{\mbf{c}} / \langle I \rangle$. One can reinterpret Lemma \ref{prop:progen} as saying that $\H_{\mbf{c}}(\ba,{\Omega})$ is a \textit{split} Azumaya algebra over $\OmCh_{\ba,{\Omega}}$, with splitting bundle $\Delta (p,\lambda)$. 
\end{remark}

\subsection{Lagrangian subvarieties}

It is shown in \cite[Proposition 4.5]{IainSurvey} that $X_{\mbf{c}}$ is a symplectic variety, see \cite{FuSurvey} for the definition and properties of symplectic varieties. This implies that the smooth locus $X_{\mbf{c}}^{\sm}$ is a symplectic leaf in $X_{\mbf{c}}$ and hence its compliment has codimension at least two in $X_{\mbf{c}}$. 

\bdefn
A reduced subvariety $Y$ of $X_{\mbf{c}}$ is said to a Lagrangian subvariety if $Y_\sm^i \cap X_{\mbf{c}}^{\sm}$ is a \textit{non-empty} Lagrangian submanifold of $X_{\mbf{c}}^{\sm}$ for each irreducible component $Y^i$ of $Y$.  
\edefn

The goal of this subsection is to prove the following proposition. It is a consequence of Gabber's Integrability Theorem \cite{Gabber}.  

\begin{prop}\label{prop:lagrangian}
Let $\ba \in \h^* / W$. Then $\pi^{-1}(\ba)_{\red}$ is a Lagrangian subvariety of $X_{\mbf{c}}$.
\end{prop}

Let $\mf{m}_{\ba}$ be the maximal ideal in $\C[\h^*]^W$ corresponding to $\ba$. First, we show

\begin{lem}\label{lem:int}
Let $J$ be the radical of the ideal generated by $\mf{m}_{\ba}$ in $\ZH_{\mbf{c}}$. Then, $J$ is involutive i.e. $\{ J, J \} \subseteq J$. 
\end{lem}

\begin{proof}
It is clear from the definition of the Poisson bracket on $\ZH_{\mbf{c}}$ that the ideal generated by $\mf{m}_{\ba}$ in $\ZH_{\mbf{c}}$ is involutive. However, it seems that this does not in general imply that $J$ is involutive. Therefore, we need to work a bit harder. Since we have not assumed any smoothness condition on $X_{\mbf{c}}$, we are also unable to use results from previous sections. Let $p_1, \ds, p_{\ell} \in \h^*$ be the elements in the orbit $\ba$. Set
$$
M = \bigoplus_{i = 1}^{\ell} \left( \bigoplus_{\lambda \in W_{p_i}} \Delta(p_i,\lambda) \right).
$$
Let $Y = (\Supp \ M)_{\red}$. Then, I claim that $Y = V(J)$. Since $\mf{m}_{\ba} \cdot M = 0$, we have $Y \subset V(J)$. Let $x \in V(J)$ and choose some simple $\H_{\mbf{c}}$-module $L$ supported on $x$. As a $\C[\h^*]$-module, $L = \oplus_{i = 1}^{\ell} L_{p_i}$, where $L_{p_i}$ is supported at $p_i$. Without loss of generality, we may assume that $L_{p_1} \neq 0$. Let $\lambda \subset L_{p_1}$ be an irreducible $W_{p_1}$-module in the socle of $L_{p_1}$. Then, there is a non-zero homomorphism $\Delta(p_1, \lambda) \rightarrow L$. This implies that $\Hom_{\H_{\mbf{c}}}(M,L) \neq 0$ and hence $x \in Y$. 

Now, let $\C[\epsilon]$ be functions on the 3rd infinitesimal neighborhood of $0$ in $\C$, so that $\epsilon^3 = 0$. It will be easier to work, via the Satake isomorphism, with the spherical subalgebra $e \H_{\mbf{c}} e$. The usual rational Cherednik algebra $\H_{t,\mbf{c}}$ has an additional parameter $t$, which we have assume throughout is set to zero. Specializing instead to $t = \epsilon$, we have a $\C[\epsilon]$-algebra $e \H_{\epsilon,\mbf{c}} e$ such that $e \H_{\epsilon,\mbf{c}} e/ \epsilon e \H_{\epsilon,\mbf{c}} e \simeq e \H_{\mbf{c}} e$. Then the Poisson structure on $e \H_{\mbf{c}} e$ is constructed as in \cite{Gabber}. We can define $\Delta_{\epsilon}(p,\lambda)$ in the obvious way. It is a $\H_{\epsilon, \mbf{c}}$-module, free over $\C[\epsilon]$. This gives us a $e \H_{\epsilon,\mbf{c}} e$-module $e M_{\epsilon}$, free over $\C[\epsilon]$. This freeness implies that condition (1.2) of \cite{Gabber} is satisfied. Then, the fact that $J$ is involutive is a consequence of \cite[Theorem II]{Gabber}, together with the fact that the Satake isomorphism is an isomorphism of Poisson algebras.     
\end{proof} 

\begin{proof}[Proof of Proposition \ref{prop:lagrangian}]
Let $n = \dim \h$. The maximal ideal $\mf{m}_{\ba}$ in $\C[\h]^W$ is generated by a regular sequence $f_1, \ds, f_n$. By Lemma \ref{lem:basicprop} (1), they also form a regular sequence in $\ZH_{\mbf{c}}$. Therefore, the fact that the morphism $\pi$ is flat, together with \cite[Corollary 9.6 (ii)]{Hartshorne}, implies that each irreducible component of $\pi^{-1}(\ba)_{\red}$ is $n$-dimensional. Let $Y^1, \ds, Y^k$ be these irreducible components. Let $\delta := \prod_{s \in \mc{S}} \alpha_s \in \C[\h]$; it is a semi-invariant with respect to $W$. Therefore, there exists some $k \in \N$ such that $\delta^k \in \C[\h]^W$. The sequence $f_1, \ds, f_n$ extends to a regular sequence $f_1, \ds, f_n, \delta^k$ in $Z_{\mbf{c}}$. This implies that, for each $i$, $\dim Y^i \cap V(\delta^k) = n-1$ if $Y^i \cap V(\delta^k) \neq \emptyset$. In particular, $Y^i \backslash V(\delta^k) \neq \emptyset$ for all $i$. The Dunkl embedding, as explained in \cite[\S 4]{EG}, shows that $\H_{\mbf{c}}[\delta^{-k}] \simeq \C[\h \times \h^*_{\reg}] \rtimes W$ where $\h^*_{\reg}$ is the set of points in $\h^*$ with trivial $W$-stabilizer. Since the centre of $\C[\h \times \h^*_{\reg}] \rtimes W$ is a regular domain, it follows that $Y^i \backslash V(\delta^k) \subset X_{\mbf{c}}^{\sm}$. Thus, we have shown that $Y^i \cap X_{\mbf{c}}^{\sm}$ is non-empty for all components $Y^i$ of $\pi^{-1}(\ba)_{\red}$. Let $J^i$ be the ideal defining $Y^i$. It is a minimal prime over the ideal $J$. By Lemma \ref{lem:int}, $J$ is an involutive ideal. Therefore, $J^i$ is an involutive ideal, see \cite[Lemma 2.1]{CoutinoKrull}. Choose some point $x \in Y^i \cap X_{\mbf{c}}^{\sm}$. Then, \cite[Proposition 1.5.1]{CG} says that $T_x Y^i$ is a coisotropic subspace of $T_x X_{\mbf{c}}$. But $\dim T_x Y = n$, therefore it is actually a Lagrangian subspace as required.   
\end{proof}

\begin{remark}
\begin{enumerate}
\item If $X_{\mbf{c}}$ is a smooth variety, as will be the case for the symmetric group, then it is a symplectic manifold and Proposition \ref{prop:lagrangian} shows that $\pi^{-1}(\ba)_{\red}$ is the disjoint union of finitely many Lagrangian submanifolds of $X_{\mbf{c}}$.
\item An analogous result to Proposition \ref{prop:lagrangian} holds when one considers $\varpi^{-1}(\bb)$ for $\bb \in \h / W$.
\end{enumerate}
\end{remark}

\section{$\Tor$ and $\Ext$}\label{sec:BV}

In this section we sketch the proof of the results stated in the introduction. A quantizable $\H_{\mbf{c}}$-module $M$ is said to be \tit{simple}, resp. \textit{generically simple}, if $e M$ is a cyclic $\ZH_{\mbf{c}}$-module, resp. $e M |_{X_{\mbf{c}}^{\sm}}$ is a cyclic $\ZH_{\mbf{c}}  |_{X_{\mbf{c}}^{\sm}}$-module. 

\begin{proof}[Proof of Proposition \ref{prop:BG}]
It will be easier to consider $\H_{\mbf{c}}$ as a sheaf of algebras on $X_{\mbf{c}}$. If $\widehat{\H}_{\mbf{t},\mbf{c}}$ denotes the completion of $\H_{\mbf{t},\mbf{c}}$ with respect to $(\mbf{t})$, then via algebraic microlocalization, $\widehat{\H}_{\mbf{t},\mbf{c}}$ is also a sheaf of algebras on $X_{\mbf{c}}$.  On the smooth locus $X^{\sm}_{\mbf{c}}$, $E$ defines an equivalence $\Lmod{\H_{\mbf{c}}} \stackrel{\sim}{\rightarrow} \Lmod{\ZH_{\mbf{c}}}$. This extends to an equivalence $\Lmod{\widehat{\H}_{\mbf{t},\mbf{c}}} \stackrel{\sim}{\rightarrow} \Lmod{e\widehat{\H}_{\mbf{t},\mbf{c}} e}$.

Since the support of $M |_{X^{\sm}_{\mbf{c}}}$ equals the support of $M' |_{X^{\sm}_{\mbf{c}}}$, and we have assumed that $M$ and $N$ have smooth intersection, 
$$
\Tor^{\H}_{\idot}(M',N)  = \Tor^{\H}_{\idot}(M' |_{X^{\sm}_{\mbf{c}}},N |_{X^{\sm}_{\mbf{c}}}) \simeq \Tor^{\ZH}_{\idot}(e M |_{X^{\sm}_{\mbf{c}}}, eN |_{X^{\sm}_{\mbf{c}}}),
$$
and, similarly, $\Ext_{\H}^{\idot}(M,N)$ is isomorphic to $\Ext^{\idot}_{\ZH} ( eM |_{X_{\mbf{c}}^{\sm}}, eN |_{X_{\mbf{c}}^{\sm}})$. Since $e M |_{X^{\sm}_{\mbf{c}}}$ and $e N |_{X^{\sm}_{\mbf{c}}}$ are simple, they are quotients of $\ZH_{\mbf{c}} |_{X^{\sm}_{\mbf{c}}}$. Therefore, they are naturally commutative algebras. Now the claims of Proposition  \ref{prop:BG} are consequences of the theory developed in \cite{GBVGinzburg}; in particular, Corollary 1.1.3.  The assumptions of \textit{loc. cit.} that $e M |_{X^{\sm}_{\mbf{c}}}$ and $e N |_{X^{\sm}_{\mbf{c}}}$ define smooth, closed subvarities of $X^{\sm}_{\mbf{c}}$ is unnecessary in our case since that assumption is only required in order to guarantee the existence of a third order quantization of the modules. But, in our case, the existence of the quantization is built into the definition of quantizable modules.   
\end{proof}

Fix $p \in \h^*$, $\Omega \in \CM_{\mbf{c}'}(W_p)$ and let $\ba$ be the image of $p$ in $\h^* / W$. Let $I$ be the ideal defining the closed subscheme $\OmCh_{\ba,\Omega}$ in $X_{\mbf{c}}$. We assume $\Omega = \{ \lambda_{\Omega} \}$, so that $\OmCh_{\ba,\Omega}$ is smooth. Then $\NN_{\ba,\Omega}^{\vee} := I / I^2 $ is a free $\ZH_{\mbf{c}} / I$-module. As noted in the introduction, it is the module of sections of the conormal bundle of $\OmCh_{\ba,\Omega}$ in $X_{\mbf{c}}$. Its dual $\NN_{\ba,\Omega} := (I / I^2)^{\vee}$ is the module of sections of the normal bundle of $\OmCh_{\ba,\Omega}$ in $X_{\mbf{c}}$. 

\begin{proof}[Proof of Corollary \ref{thm:extiso1}]
Let $\lambda^* = \Hom_{\C}(\lambda,\C)$, an irreducible right $W_p$-module. We denote by $(p,\lambda^*) \Delta$ the right $\H_{\mbf{c}}$-module induced from the right $\C[\h^*] \rtimes W_p$-module $\lambda^*$. Standard arguments using the Kozsul resolution show that $\Ext^n_{\H_{\mbf{c}}}(\Delta(p,\lambda),\H_{\mbf{c}}) \simeq (p,\lambda^*) \Delta$ c.f. \cite[Lemma 4.1]{GGOR}. The results of Theorem \ref{thm:main} apply equal well to the right $\H_{\mbf{c}}$-module $(p,\lambda^*) \Delta$. In order to be able to apply Proposition \ref{prop:BG}, the only thing left to check is that $e \Delta(p,\lambda)$ and $(p,\lambda^*) \Delta e$ define the same smooth, closed subvariety of $X_{\mbf{c}}$ i.e. that they are isomorphic as $\ZH_{\mbf{c}}$-modules. By completing $\H_{\mbf{c}}$ at $\bb \in \h^* / W$, we may assume that $p = 0$. Then, $\Ext^n_{\H_{\mbf{c}}}(\Delta(p,\lambda),\H_{\mbf{c}}) \simeq (p,\lambda^*) \Delta$ implies that the support of $ (p,\lambda^*) \Delta$ is contained in the support of $\Delta(p,\lambda)$. In particular, it is contained in the smooth locus of $X_{\mbf{c}}$. Thus, the "right block" of $\H^{\bo}_{\mbf{c}}$ containing $\lambda^*$ consists of a single element and Corollary \ref{cor:polyring} implies that the support of  $ (p,\lambda^*) \Delta$ is isomorphic to $\mathbb{A}^n$. Thus, $e \Delta(p,\lambda) \simeq (p,\lambda^*) \Delta e$ as $\ZH_{\mbf{c}}$-modules. The statement of the corollary now follows from Theorem \ref{thm:BViso} of the appendix. 
\end{proof}

Recall from Theorem \ref{thm:deltaend} that the graded character of $\C[\OmCh_{\bo,\Omega}]$ is $q^{-b_{\Omega }}f_{\Omega}(q) \prod_{i = 1}^n (1 - q^{d_i})^{-1}$, where $d_1, \ds, d_n$ are the degrees of $(W,\h)$. Since $\C[\OmCh_{\bo,\Omega}]$ is a polynomial ring, this implies that there exists\footnote{For an arbitrary representation $\lambda \in \Irr (W)$, it is not possible to find integers $0 < e_1 \le \ds \le e_n$ such that (\ref{eq:eis}) holds. This is related to the fact that $\bx_{\lambda}$ is always in the singular locus of $X_{\mbf{c}}$, regardless of the parameter $\mbf{c}$.} integers $0 < e_1 \le \ds \le e_n$ such that 
\beq{eq:eis}
q^{-b_{\Omega}} f_{\Omega}(q) \prod_{i = 1}^n \frac{1}{1 - q^{e_i}} = \prod_{i = 1}^n \frac{1}{1 - q^{d_i}}.
\eeq

\begin{cor}
When $\Omega = \{ \lambda_{\Omega} \}$, we have, as bigraded vector spaces, 
\beq{eq:torchar}
\mathrm{ch}_{q,t}  \left( \Tor_{\idot}^{\H_{\mbf{c}}}(\Delta(\lambda)^{\op},\Delta(\lambda)) \right) = q^{-b_{\Omega}}f_{\Omega}(q) \prod_{i = 1}^n \frac{1 + t q^{-e_i}}{1 - q^{d_i}},
\eeq
and
\beq{eq:extchar}
\mathrm{ch}_{q,t} \left( \Ext^{\idot}_{\H_{\mbf{c}}}(\Delta(\lambda),\Delta(\lambda)) \right) = q^{-b_{\Omega}}f_{\Omega}(q) \prod_{i = 1}^n \frac{1 + t q^{e_i}}{1 - q^{d_i}} . 
\eeq
\end{cor}

\begin{proof}
Let $x$ be the $\Cs$-fixed point in $X_{\mbf{c}}$ over which $L(\lambda_{\Omega})$ lives. As in the proof of Corollary \ref{cor:polyring}, the integers $e_i$ are defined so that the graded character of $T_{x} \Lambda_{\bo,\Omega}$ is given by $\sum_{i = 1}^n q^{-e_i}$. We have $T_{x} X_{\mbf{c}} = T_{x} \Lambda_{\bo,\Omega} \oplus (T_{x} \Lambda_{\bo,\Omega})^{\perp}$ with respect to the symplectic form on $X_{\mbf{c}}$, and we may identify $(T_{x} \Lambda_{\bo,\Omega})^{\perp} = (T_{\Lambda_{\bo,\Omega}} X_{\mbf{c}})_{x}$ as $\Cs$-representations. This implies that 
$$
\mathrm{ch}_q (T_{\Lambda_{\bo,\Omega}} X_{\mbf{c}})_{x} = \sum_{i = 1}^n q^{e_i}.
$$
Then formula (\ref{eq:extchar}) follows from the fact that 
$$
\Ext^{\idot}_{\H_{\mbf{c}}}(\Delta(\lambda_{\Omega}),\Delta(\lambda_{\Omega})) = \C[\Lambda_{\bo,\Omega}] \o \wedge^{\idot} (T_{\Lambda_{\bo,\Omega}} X_{\mbf{c}})_{x}
$$
as graded vector spaces, together with the fact that if $V$ is a graded vector space with character $q^{m_1} + \cdots + q^{m_k}$ then the bigraded character of $\wedge^{\idot} V$ is given by $(1 + t q^{m_1}) \cdots (1 + tq^{m_k})$. The proof of formula (\ref{eq:torchar}) is similar. 
\end{proof}

We end the section by noting one other situation were one can easily calculate virtual deRham (co)homology groups. Let $q \in \h$ and $\mu \in \Irr (W_q)$. We let the dual Verma module associated to $q$ and $\mu$ be the induced module $\nabla (q,\mu) := \H_{\mbf{c}} \o_{\C[\h] \rtimes W_q} \mu$, where the action of $\C[\h]$ on $\mu$ is via evaluation at $q$. Choose $ \Omega \in \CM_{\mbf{c}}(W)$ such that $\Omega = \{ \lambda_{\Omega} \}$.   Let $\C[i]$ denote the complex with a copy of $\C$ in degree $i$ and zero in all other degrees. 

\begin{prop} 
We have $\Tor^{\H_{\mbf{c}}}_{\idot}(\Delta(\lambda_{\Omega}),\nabla(\lambda_{\Omega})) \simeq \C[0]$ and $\Ext_{\H_{\mbf{c}}}^{\idot}(\Delta(\lambda_{\Omega}),\nabla(\lambda_{\Omega})) \simeq \C[n]$.
\end{prop}

\begin{proof}
Our assumptions on $\Omega$ imply that there is a regular sequence of homogeneous elements $\bz = z_1, \ds, z_n$ such that $\Lambda_{\bo,\Omega}  = V(\bz)$. Let $\Pi_{\Omega}$ denote the support of $\nabla (\lambda_{\Omega})$. The analogue of Theorem \ref{thm:main} in this context implies that $\Pi_{\Omega}$  is a smooth Lagrangian subvariety of $X_{\mbf{c}}$, isomorphic to $\mathbb{A}^n$. Also, $\Tor^{\H}_{\idot}(\Delta(\lambda_{\Omega}),\nabla(\lambda_{\Omega})) \simeq \Tor^{X_{\mbf{c}}}_{\idot}\left( \mc{O}_{\Lambda_{\bo,\Omega}},\mc{O}_{\Pi_{\Omega}} \right)$ and $\Ext_{\H}^{\idot}(\Delta(\lambda_{\Omega}),\nabla(\lambda_{\Omega})) \simeq \Ext_{X_{\mbf{c}}}^{\idot}\left( \mc{O}_{\Lambda_{\bo,\Omega}},\mc{O}_{\Pi_{\Omega}} \right)$, so it suffices to calculate the latter groups. It is explained in chapter 17 of \cite{Eisenbud} how one does this. We use the notation of \textit{loc. cit.} Using the fact that both $\Lambda_{\bo,\Omega}$ and $\Pi_{\Omega}$ are smooth subvariety of $X_{\mbf{c}}$ intersecting transversely in a single point, we have 
$$
\Tor^{X_{\mbf{c}}}_{\idot}\left( \mc{O}_{\Lambda_{\bo,\Omega}},\mc{O}_{\Pi_{\Omega}} \right) = H_{\idot}(\bz,\mc{O}_{\Pi_{\Omega}}) \simeq \C[0],
$$
and $\Ext_{X_{\mbf{c}}}^{\idot}\left( \mc{O}_{\Lambda_{\bo,\Omega}},\mc{O}_{\Pi_{\Omega}} \right) = H^{\idot}(\bz,\mc{O}_{\Pi_{\Omega}}) = H_{n - \idot}(\bz,\mc{O}_{\Pi_{\Omega}}) \simeq \C[n]$, as required. 
\end{proof}

Obviously, $H_{\mathrm{vir}}^{\idot} (\Delta(\lambda_{\Omega}),\nabla(\lambda_{\Omega})) = \C[n]$ and $H^{\mathrm{vir}}_{\idot} ((\lambda_{\Omega}^*)\Delta,\nabla(\lambda_{\Omega})) = \C$ in this case. 

\section{Remarks}

By Lemma \ref{claim:idempotentlifting} (5), the connected components of $\varpi^{-1}(\bo)_{\red}$ are naturally labeled by the elements of $\CM_{\mbf{c}}(W)$. Moreover, Lemma \ref{lem:basicprop} implies that the irreducible components of $\varpi^{-1}(\bo)_{\red}$ are $\dim \h$-dimensional. Based on examples, it seems natural to conjecture: 

\begin{enumerate}
\item Each connected component of $\varpi^{-1}(\bo)_{\red}$ is irreducible and isomorphic to $\mathbb{A}^{n}$, where $n = \dim \h$; i.e. $\varpi^{-1}(\bo)_{\red} = \bigsqcup_{\Omega} \mathbb{A}^n_{\Omega}$.
\item For each $\Omega \in \CM_{\mbf{c}}(W)$,  $E_{\Omega} = \C[\mathbb{A}^n_{\Omega}]$. 
\item The map $\pi | : \mathbb{A}^n_{\Omega} \rightarrow \h / W$ is flat of degree $\dim \lambda_{\Omega}$, and is generically a covering. \\

\noindent We note that if $Y$ is the connected component labeled by $\Omega$, then one can study the generic fibers of $\pi | : Y \rightarrow \h / W$ by considering the representation theory of $\H_{\mbf{c}}^{\bb}$. Namely, statements (2) and (3) above would imply the following: \\

\item Assume that $\bb$ is generic. Then the head of $P^+(\lambda_{\Omega})^{\bb}$ equals $\Delta(0,\lambda_{\Omega},\bb)$, a semi-simple module, with pairwise non-isomorphic simple summands.  
\end{enumerate}

\section{Appendix: Batalin-Vilkoviski structures}

In this appendix we summarize the main results of \cite{GBVGinzburg}, as required in this article. We follow the presentation of \textit{loc. cit.}, which the reader is encouraged to consult for further details. All undecorated tensor products will mean tensor product over $\C$.  

Let $D = \bigoplus_{i \ge 0} D_i$ be a graded commutative algebra. If $D$ is equipped with a differential $\delta : D_{\bullet} \rightarrow D_{\bullet -1}$, $\delta^2 = 0$, and\footnote{This equation is equivalent to the fact that $\delta$ is a differential operator of order at most two.}
\begin{align*}
\delta(abc)  = & \ \delta(ab) c + (-1)^{\deg(a)} a \delta(bc) + (-1)^{(\deg(a) +1) \deg(b)} b \delta(ac) - \delta(a) bc \\
 & - (-1)^{\deg(a)} a \delta(b) c - (-1)^{\deg(a) + \deg(b)} a b \delta(c) + \delta(1) abc,
\end{align*}
for all homogeneous elements $a,b,c$ of $D$, then $D$ is said to be a \textit{Batalin-Vilkoviski} (BV) algebra. Every Batalin-Vilkoviski algebra has in addition the structure of a \textit{Gerstenhaber} algebra. Namely, for $x,y$ homogeneous elements in $D$, the formula
$$
[x,y] := \delta(xy) - \delta(x)y - (-1)^{\deg x} x \delta(y)
$$
makes $D_{\bullet}$ into a Gerstenhaber algebra. This means that $[ - , - ]$ has degree $-1$ and 
\begin{align*}
[a,bc] & = [a,b] c + (-1)^{(\deg(a) - 1) \deg (b)} b [a,c],\\
[a,b] & = -(1)^{(\deg(a) -1)(\deg(b) - 1)} [b,a], \\
[a,[b,c]] & = [[a,b],c] + (-1)^{(\deg(a) -1)(\deg(b) - 1)} [b,[a,c]],
\end{align*}
for all homogeneous elements $a,b,c$ of $D$ i.e. the graded commutative algebra $D$ is an odd Poisson algebra. 

Let $X$ be a smooth affine variety equipped with a symplectic two-form $\omega$. Let $Y$ be a smooth, coisotropic subvariety of $X$. Let $\NN_{X/Y}$ denote the sheaf of sections of the normal bundle of $Y$ in $X$. It is a sheaf of $\mc{O}_X$-modules supported on $Y$. Its dual $\NN_{X/Y}^{\vee}$ is the sheaf of sections of the conormal bundle. Since the two-form $\omega$ is non-degenerate, it induces an isomorphism $\Gamma(X,\Omega^2_X) \simeq \Gamma(X,\wedge^2 \Theta_X)$. We let $P$ be the image of $\omega$ under this isomorphism; it is a Poisson bivector. The graded commutative algebras $\wedge^{\bullet} \NN_{X/Y}^{\vee}$ and $\wedge^{\bullet} \NN_{X/Y}$ have a natural structure of BV algebra. Namely, the differential on $\wedge^{\bullet} \NN_{X/Y}^{\vee}$ is given by the formula $\delta = i_P \circ d_{\mathsf{DR}} + d_{\mathsf{DR}} \circ i_P$, where $d_{\mathsf{DR}}$ is the deRham differential of degree one and $i_P$ denotes contraction with $P$ (and thus $\delta$ has degree $-1$). The differential on $\wedge^{\bullet} \NN_{X/Y}$ is given by the Schouten bracket $[P,-]$. That these operations are indeed well-defined follows from the fact that the ideal defining $Y$ is involutive. Combining Corollary 1.1.3, Propositions 5.1.1 and 5.2.1 of \cite{GBVGinzburg} gives    

\begin{thm}\label{thm:BViso}
The sheaf of graded algebras $\mc{T}or_{\bullet}^{\mc{O}_X}(\mc{O}_Y,\mc{O}_Y)$ admits a canonical structure of a Gerstenhaber algebra such that 
\beq{eq:BViso1}
\mc{T}or_{\idot}^{\mc{O}_X}(\mc{O}_Y,\mc{O}_Y) \simeq \wedge^{\idot} \NN_{X/Y}^{\vee}
\eeq
as Gerstenhaber algebras. 

Moreover, the sheaf $\mc{E}xt^{\idot}_{\mc{O}_X}(\mc{O}_Y,\mc{O}_Y)$ admits a canonical structure of Gerstenhaber module over the Gerstenhaber algebra $\mc{T}or_{\idot}^{\mc{O}_X}(\mc{O}_Y,\mc{O}_Y)$ such that 
\beq{eq:BViso2}
\mc{E}xt^{\idot}_{\mc{O}_X}(\mc{O}_Y,\mc{O}_Y) \simeq \wedge^{\idot} \NN_{X/Y}
\eeq
as Gerstenhaber modules, compatible in the obvious sense with the identification (\ref{eq:BViso1}). 
\end{thm}

We denote by $H^{\mathrm{vir}}_{\idot}(\mc{O}_Y,\mc{O}_Y)$, resp.  $H_{\mathrm{vir}}^{\idot}(\mc{O}_Y,\mc{O}_Y)$, the cohomology of $(\Tor^{\mc{O}_X}_{\idot}(\mc{O}_Y,\mc{O}_Y), \delta)$, resp. of $(\Ext^{\idot}_{\mc{O}_X}(\mc{O}_Y,\mc{O}_Y),\delta)$. Then, results of Brylinski \cite[Corollary 2.2.2]{BrylinskPoisson} and Lichnerowicz \cite[Theorem 2.1.4]{DufourZung} imply:

\begin{cor}\label{cor:virtual}
We have $H^{\mathrm{vir}}_{\idot}(\mc{O}_Y,\mc{O}_Y) \simeq H^{\dim X - \idot}_{\mathsf{DR}}(Y)$ and $H_{\mathrm{vir}}^{\idot}(\mc{O}_Y,\mc{O}_Y) \simeq H^{\idot}_{\mathsf{DR}}(Y)$.
\end{cor}

\def\cprime{$'$} \def\cprime{$'$} \def\cprime{$'$} \def\cprime{$'$}
  \def\cprime{$'$} \def\cprime{$'$} \def\cprime{$'$} \def\cprime{$'$}
  \def\cprime{$'$} \def\cprime{$'$} \def\cprime{$'$} \def\cprime{$'$}
  \def\cprime{$'$}


\end{document}